\documentclass[final,leqno,onefignum,onetabnum]{siamltex1213}
\usepackage{times}
\usepackage{amsmath,amsfonts,amssymb,latexsym,epsfig}
\usepackage{color,epsf}
\usepackage{graphicx}
\newcommand{\T}{^{\mathsf{T}}}
\title{Symplectic Runge-Kutta schemes for adjoint equations, automatic differentiation, optimal control and more \thanks{This research is supported by
projects MTM2010-18246-C03-01 and MTM2013-46553-C3-1-P from Mi\-nisterio de Ciencia e
Innovaci\'on, Spain.}}

\author{
J.M.\ Sanz-Serna\thanks{Departamento de Matem\'aticas, Universidad Carlos III de Madrid, Avenida de la Universidad 30, 28911 Legan\'es (Madrid), Spain.
\email{jmsanzserna@gmail.com}}
}
\begin{document}
\maketitle
\slugger{sirev}{xxxx}{xx}{x}{x--x}

\begin{abstract} The study of the sensitivity of the solution of a system of differential equations with respect to changes in the initial conditions leads to the introduction of an adjoint system, whose discretisation is related to reverse accumulation in automatic differentiation. Similar adjoint systems arise in optimal control and other areas, including classical Mechanics. Adjoint systems are introduced in such a way that they exactly preserve a relevant {\em quadratic invariant} (more precisely an inner product). Symplectic Runge-Kutta and Partitioned Runge-Kutta methods are defined through the exact conservation of a differential geometric structure, but may be characterized by the fact that they  preserve exactly {\em quadratic invariants}  of the system being integrated. Therefore the symplecticness (or lack of symplecticness) of a Runge-Kutta or Partitioned Runge-Kutta integrator should be relevant to understand its performance when applied to the computation of sensitivities, to optimal control problems and in other applications requiring the use of adjoint systems. This paper examines the links between symplectic integration and those applications.
The article presents in a new, unified way a number of results now scattered or implicit in the literature. In particular we show how some common procedures, such as the direct method in optimal control theory and the computation of sensitivities via reverse accumulation, imply, probably unbeknownst to the user, \lq hidden\rq\ integrations with symplectic Partitioned Runge-Kutta schemes.
\end{abstract}

 \begin{keywords}Runge-Kutta methods, Partitioned Runge-Kutta methods, symplectic integration, Hamiltonian systems, variational equations, adjoint equations, computation of sensitivities, Lagrange multipliers, automatic differentiation, optimal control, Lagrangian mechanics, reflected and transposed Runge-Kutta schemes, differential-alg\-ebraic problems, constrained controls\end{keywords}

 \begin{AMS}34H05, 49A10, 65L06, 65K10, 65P10, 70H25\end{AMS}

\pagestyle{myheadings}
\thispagestyle{plain}
\markboth{J.~M. SANZ-SERNA}{SYMPLECTIC RK SCHEMES FOR ADJOINTS, CONTROL AND MORE}

\section{Introduction}

Symplectic Runge-Kutta (RK) \cite{lasagni}, \cite{bit}, \cite{suris} and Partitioned Runge-Kutta (PRK) \cite{abia2}, \cite{suris2} formulae were introduced to integrate Hamiltonian systems in long time intervals. They are defined in terms of a purely geometric property, the conservation of the symplectic structure, and provided the first widely studied instance of what  was later termed {\em geometric integration} \cite{geom}. It is well known that symplectic RK methods may be characterized as being those that exactly preserve all {\em quadratic} first integrals (invariants of motion) of the system being integrated.  This is a useful property: for instance the (symplectic) implicit midpoint rule is sometimes chosen to integrate wave equations because it conserves quadratic invariants.
However quadratic conservation has taken a back seat to the symplectic property itself in the geometric integration literature.
The aim of this paper is to emphasize that the conservation of quadratic invariants plays an important role in the computation of numerical sensitivities, in optimal control theory and in classical mechanics. In all these areas there is an
interplay between variational equations and their adjoints, an interplay  based {\em on the conservation of a key quadratic invariant} (see (\ref{eq:prop})). The conservation of this invariant gives relevance to the symplecticness of the integrator.
Actually, some widely used procedures, such as the direct method in optimal control theory and the computation of sensitivities via reverse accumulation, imply \lq hidden\rq\ integrations with symplectic PRK schemes; therefore the theory of symplectic PRK integration should be helpful in understanding such procedures. From a more abstract point of view one may say that the purpose of this article is to clarify the behaviour of RK  integrators {\em vis-\`{a}-vis} the operation of taking {\em adjoints}: an RK  method is symplectic precisely if it commutes with the formation of adjoints.

The paper presents a coherent treatment of results spread across the
literature of various communities together with some new, unifying results.
In order to cater for a variety of possible readers, this article is written without assuming much background. We hope it will help researchers in optimal control to better understand RK schemes and, similarly,  encourage RK experts to consider sensitivities and optimal control problems.

Section \ref{sec:rk} provides background on numerical integrators. We introduce the necessary notation and recall a number of properties of symplectic RK and related schemes. In particular, we quote some results (Theorems \ref{th:cooper}, \ref{th:cooper2}) that ensure the exact preservation by the integrator of quadratic conservation laws.

Section \ref{s:adjoint}, the core of the paper, is devoted to the integration of the adjoint variational equations used to perform sensitivity analysis. It is well known that an RK method $\cal M$ applied to  the variational equations of a  system $\cal S$ automatically produces the variational equations for the  discretisation of $\cal S$ by means of $\cal M$ (Theorem \ref{th:variational}); in other words, the operation of RK discretisation {\em commutes} with the operation of forming variational equations.
The situation for the adjoints is more complicated, cf.\ \cite{sirkes}, because commutation will only take place if the discretisation is carried out so as to {\em exactly conserve the key quadratic invariant} (\ref{eq:prop}) and, in some way, this demands a symplectic integrator. There are three cases of increasing complexity:
\begin{itemize}
\item $\cal S$ is integrated with a {\em symplectic} RK scheme $\cal M$. Then the application of $\cal M$ to the adjoint equations of $\cal S$ produces the adjoint equations for  discretisation of $\cal S$ by means of $\cal M$ (Theorem \ref{th:adjoint}).
\item $\cal S$ is integrated with a {\em non-symplectic} RK scheme $\cal M$ whose weights do not vanish. Then,  the adjoint equations for the discretisation are obtained by integrating the adjoint equations of $\cal S$ with a {\em different} set of RK coefficients, so that the overall procedure is a symplectic PRK method (Theorem \ref{th:adjointPRK}). The recipe for the adjoint coefficients is given in formula (\ref{eq:trick}) below. The  method used for the adjoint equations will in general be of lower order than the RK scheme $\cal M$ used for the main integration and  will also have different stability properties.  For these reasons non-symplectic methods $\cal M$ should be used with care. The computation of sensitivities of the discrete solution via {\em automatic differentiation with reverse accumulation} implicitly provides the {\em symplectic PRK integration} of the adjoint equations with coefficients (\ref{eq:trick}) (Theorem \ref{th:automatica}).
\item $\cal S$ is integrated with a {\em non-symplectic} RK scheme $\cal M$ having one or more null weights. Then, to obtain the adjoint equations of the discretisation, the continuous adjoint equations have to be integrated with a fancy integrator outside the RK class (see the appendix). Again an order reduction is likely to take place and again the fancy integration is implicitly performed whenever differentiation with reverse accumulation is used.
\end{itemize}

Section \ref{sec:control} deals with the Mayer optimal control problem in the case of unconstrained controls. There is again a quadratic conservation law that is of crucial importance and this fact brings symplectic schemes to the foreground.
The results there are quite similar to those in the preceding section
(the case of vanishing weights is discussed in the appendix):
\begin{itemize}
\item For a symplectic RK method, {\em commutation} \cite{ross} takes place : the discretisation of the continuous first order  conditions necessary for optimality provides the  first order necessary conditions for the discrete solution (Theorem \ref{th:main2}).
\item When the equations for the states are discretised with a {\em non-symplectic} RK scheme {\em with non-vanishing weights}, to achieve commutation  the costate equations have to be integrated by means of a clever set of coefficients that does not coincide with the set used for the states (Theorem \ref{th:main2}).  With this clever set, the overall integration (states+costates) is performed with a symplectic PRK method. In general, an order reduction will take place for states, costates and controls. As first noted by Hager \cite{hager}, the required set of coefficients is alternatively defined, not by imposing symplecticness of the integration, but by using the {\em direct} approach, i.e. by minimising the cost in the discrete realm with the help of Lagrange multipliers (Theorem \ref{th:direct}).
\end{itemize}

For a {\em symplectic}  RK or PRK integration of the system for states and costates, the direct and indirect approach are mathematically equivalent. When a  non-symplectic PRK is used in the indirect approach, the discrete solution {\em cannot} be reached via the direct approach, which always implies a symplectic integration of the states+costates system.

Extensions to more general control problems are presented in Section \ref{sec:exten}. Section \ref{s:mech} is devoted to classical mechanics. Hamilton's variational principle may of course be viewed as an optimal control problem: it is a matter of minimising a functional subject to differential constraints. As is well known, the application of the theory of optimal control to this situation replicates the standard procedure to obtain Hamilton's canonical equations from Hamilton's principle. In the discrete realm, this process provides the variational derivation of symplectic PRK integrators, originally due to Suris \cite{suris2}.

Section~\ref{s:scherer} relates the preceding material to the notions of reflection and transposition of RK coefficients introduced by Scherer and T\"{u}rke  \cite{scherer} and Section~\ref{s:conc} concludes.

There is an appendix that deals with the problem of how to \lq supplement\rq\ a given non-symplectic RK method with some vanishing weights so as to have a symplectic algorithm for partitioned systems.

In order not to clutter the exposition with unwanted details, I shall not be concerned  with technical issues such as existence of solutions of implicit integrators, smoothness requirements and so on. These   may be very important in some circumstances (e.g. lack of smoothness poses difficulties  if the controls are constrained, see  \cite{hager2}).

To keep the length of this work within reasonable limits I shall not discuss some other interesting connections. The duality between the Fokker-Planck equations and the Kolmogorov Backward equations in the theory of Markov stochastic processes  \cite{feller} provides another instance of the occurrence of adjoints; the material in this paper may be easily extended to study that situation. The paper \cite{frank} shows how the symplecticness of the integrator may be used to ensure symmetry-preserving simulations of the matrix Riccati equation in the feed-back representation of linear/quadratic optimal control problems.

\section{Numerical integrators}
\label{sec:rk}
In this section we  review some results on RK and related methods. For more details the reader is referred to \cite{ssc}, \cite{butcher}, \cite{hlw}, \cite{hnw}, \cite{hw}.

\subsection{Runge-Kutta schemes}
An RK method with $s$ stages is specified by $s^2+2s$  numbers
\begin{equation}\label{eq:rkabc}
a_{ij},\quad i,j = 1,\dots,s, \qquad b_i,\: c_i,\quad i=1,\dots, s.
\end{equation}
Given a $D$-dimensional differential system, $F:\mathbb{R}^D\times \mathbb{R} \rightarrow \mathbb{R}^D$,
\begin{equation}\label{eq:ode}
\frac{d}{dt} y = F(y,t),
\end{equation}
to be studied in an interval, $t_0\leq t\leq t_0+T$, and an initial condition
\begin{equation}\label{eq:ic}
\qquad y(t_0) = A\in\mathbb{R}^D,
\end{equation}
 the method (\ref{eq:rkabc}) finds approximations $y_n$ to
the values $y(t_n)$, $n = 0,1,\dots, N$, of the solution of  (\ref{eq:ode})--(\ref{eq:ic}), $ t_0 < t_1 < \cdots < t_N=t_0+T$, by setting
$y_0= A$ and, recursively,
\begin{equation}\label{eq:rkstep}
y_{n+1} = y_n +h_n \sum_{i=1}^s b_i K_{n,i}, \qquad n= 0,1\dots, N-1.
\end{equation}
Here $h_n=t_{n+1}-t_n$ denotes the step-length  and   $K_{n,i}$, $i=1,\dots,s$, are the \lq slopes\rq\
\begin{equation}\label{eq:new}
K_{n,i} = F(Y_{n,i}, t_n+c_ih_n)
\end{equation}
at the so-called internal stages $Y_{n,i}$.
The vectors $Y_{n,1}$,\dots, $Y_{n,s}$ are in turn defined by the relations
\begin{equation}\label{eq:rkstages}
Y_{n,i} = y_n + h_n \sum_{j=1}^s a_{ij} K_{n,j}, \quad i=1,\dots, s.
\end{equation}
In the particular case where the matrix $(a_{ij})$ is, perhaps after renumbering the stages, strictly lower triangular (explicit RK methods), the stages are computed recursively from (\ref{eq:new})--(\ref{eq:rkstages}).
 In the  general case,  (\ref{eq:new})--(\ref{eq:rkstages}) provides, for each $n$, a system of coupled equations to be solved for the stages.

The internal stages should not  be confused with the values $y_n$ output by the integrator and may merely be regarded as auxiliary variables. Alternatively, the vector $Y_{n,i}$ is sometimes viewed as an approximation to the off-step value $y(t_n+c_ih_n)$. It is important to emphasise that the differences $y(t_n+c_ih_n)-Y_{n,i}$ are typically much larger than the differences $y(t_n)-y_n$.

When the system (\ref{eq:ode}) is autonomous, i.e.\ $F = F(y)$, the $c_i$ play no role. At the other end of the spectrum, if $F$ is independent of $y$, the RK discretisation amounts to the use in the interval $t_0\leq t\leq t_0+T$ of the composite quadrature rule based on the {\em abscissas} $c_i$ and the {\em weights} $b_i$.

An RK scheme is said to possess order $\rho$ if, for  $t_0\leq t_n \leq t_0+T$ and smooth problems, $|y_n-y(t_n)| = \mathcal{O}(h^\rho)$, where $h = \max_n h_n$. The expansion of the local truncation error in powers of the step-length $h_n$ includes, for each power $h_n^k$, $k= 1, 2,\dots$, one or several elementary differentials of $F$; an integrator has order $\geq \rho$ if and only if, in that expansion, the coefficients of the elementary differentials of orders $k =1,\dots,\rho$ vanish. For instance, the relations (order conditions)
\begin{equation}\label{eq:ordcond1}
\sum_{i=1}^s b_i=1,\quad \sum_{i,j=1}^s b_ia_{ij}=\frac{1}{2},\quad \sum_{i,j,k=1}^s b_ia_{ij}a_{jk}=\frac{1}{6},\quad \sum_{i,j,k=1}^s b_ia_{ij}a_{ik}=\frac{1}{3},\
\end{equation}
ensure order at least $3$ for autonomous problems. They correspond to the elementary differentials $F$ (of order 1), $(\partial_y F) F$ (of order 2) and $(\partial_y F)(\partial_y F) F$,
$(\partial_{yy} F)[ F, F] $ (both of order 3)  ($\partial_y F$ is the Jacobian matrix and $\partial_{yy} F$ the tensor of second derivatives).
Since the work of Butcher in the early 1960's, order conditions and elementary differentials are studied with the help of graphs. To impose order $\geq \rho$ for autonomous problems, there is an independent order condition for each rooted tree with $\rho$ or fewer vertices. Most, but not all, useful RK schemes satisfy $c_i =\sum_j a_{ij}$ for each $i$; for them order $\rho$ for autonomous problems implies order $\rho$ for all problems.

In general RK methods do not conserve exactly the quadratic first integrals of the system being integrated. The simplest illustration is afforded by the familiar Euler's rule ($s=1$, $b_1 = 1$, $a_{11} = 0$, $c_1 = 0$) applied to the harmonic oscillator ($D=2$) $$\frac{d}{dt} y^1 = -y^2,\quad \frac{d}{dt} y^2 = y^1$$ (superscripts denote components). The (quadratic) energy $I = (1/2)((y^1)^2+
(y^1)^2)$ is conserved by the differential system because $$\frac{d}{dt} I = y^1 \frac{d}{dt}y^1+y^2
\frac{d}{dt}y^2 = y^1(-y^2)+y^2 y^1 =0.$$
However for Euler's rule it is trivial to check that, over one step,
$$
I(y^1_{n+1},y^2_{n+1})-I(y^1_{n},y^2_{n}) = \frac{h_n}{2} \big((y_n^1)^2+(y_n^2)^2\big),
$$
with an energy increase. This lack of exact preservation takes place for all explicit RK integrators, even when their order $\rho$ is high. On the other hand, it is well known and easy to prove that for the implicit midpoint rule ($s=1$, $b_1 = 1$, $a_{11} = 1/2$, $c_1 = 1/2$) and the harmonic oscillator $I(y^1_{n+1},y^2_{n+1})=I(y^1_{n},y^2_{n})$.

The present paper is based on the following 1987 result of Cooper \cite{cooper}. It ensures that {\em some} RK methods automatically inherit each quadratic conservation law possessed by the system being integrated.

 \begin{theorem}\label{th:cooper}Assume that the system (\ref{eq:ode}) possesses a quadratic first integral
 $I$, i.e.\ $I(\cdot,\cdot)$ is a real-valued bilinear mapping in $\mathbb{R}^D\times \mathbb{R}^D$ such that, for each $A$ and $t_0$, the solution $y(t)$ of (\ref{eq:ode})--(\ref{eq:ic}) satisfies
$(d/dt)I(y(t),y(t))\equiv 0$. The relations
 \begin{equation}\label{eq:sympcond}
  b_ia_{ij}+b_ja_{ji} -b_ib_j=0, \qquad i,j = 1,\dots, s,
 \end{equation}
 guarantee that, for each RK trajectory $\{y_n\}$ satisfying (\ref{eq:rkstep})--(\ref{eq:rkstages}), $I(y_{n},y_{n})$ is independent of $n$.
 \end{theorem}

We shall not reproduce here the proof of this result; it is similar to that of Theorem \ref{th:cooper2} below. The relations (\ref{eq:sympcond}) are essentially necessary for an RK scheme to conserve {\em each} quadratic first integral of  {\em each} differential system \cite[Chapter VI, Theorems 7.6, 7.10]{hlw}.

In many applications the system (\ref{eq:ode}) is Hamiltonian. This means that $D$ is even and, after writing $y=[q\T,p\T]\T$, $F=[f\T,g\T]\T$, with $q,p,f,g\in\mathbb{R}^d$, $d=D/2$, there exists a real-valued function $H(p,q,t)$ (the Hamiltonian) such that $f^r = \partial H/\partial p^r$, $g^r = -\partial H/\partial q^r$, $r = 1,\dots, d$ (superscripts indicate components). Hamiltonian systems are characterised geometrically by the symplectic property of the corresponding solution flow \cite{arnold}. When $d=1$, symplecticness means  conservation of oriented area; in higher dimensions a similar but more complicated interpretation, based on differential forms, exists; such interpretation is not required to read this paper. It is often advisable \cite{ssc}, \cite{hlw}, \cite{benbook} to integrate Hamiltonian problems by means of so-called symplectic algorithms, i.e.\ algorithms such that the transformation $y_n \mapsto y_{n+1}$ in $\mathbb{R}^{2d}$ is symplectic; those algorithms are particularly advisable in integrations where the interval $t_0\leq t \leq t_0+T$ is long (for a recent reference in that connection, see \cite{laskar}, which is part of a project to integrate the solar system over a 60 million year interval). Using the method of modified equations
\cite{griffiths}, each numerical solution may (approximately) be interpreted as a true solution of a nearby differential system called the modified system. For symplectic methods applied to Hamiltonian systems, the modified system is Hamiltonian; for non-symplectic discretisations, the modified system, while perhaps close to the system being integrated, is not Hamiltonian and this fact is likely to imply a substantial distortion of the long-time dynamics \cite{ssc}, \cite{hlw}.

The first symplectic integrators were constructed in an {\em ad hoc} way; it was later discovered (independently by Lasagni \cite{lasagni}, Suris \cite{suris} and the present author \cite{bit})  that the class of RK methods contains many symplectic schemes:
\begin{theorem}\label{th:lsss}
Assume that the system (\ref{eq:ode}) is Hamiltonian. The relations (\ref{eq:sympcond}) guarantee that the mapping $y_n \mapsto y_{n+1}$  defined in (\ref{eq:rkstep})--(\ref{eq:rkstages}) is symplectic.
\end{theorem}

The proof of Theorem \ref{th:lsss}, not included here, is very similar to the proof of  Theorem \ref{th:cooper}.
Just as for the conservation of quadratic first integrals, it turns out, see \cite{ssc}, Section 6.5, that the relations (\ref{eq:sympcond}) are essentially necessary for $y_n \mapsto y_{n+1}$  to be symplectic for each Hamiltonian system.

The set of relations (\ref{eq:sympcond}) thus ensures {\em two} different properties: quadratic conservation and symplecticness. These two properties are not unrelated:  symplecticness may be viewed a consequence of the quadratic conservation because, as noted in \cite{bochev}, the preservation of the symplectic structure by a Hamiltonian solution flow may be interpreted as a bilinear first integral of the solution flow of the associated variational system.

 The symplectic character of RK schemes satisfying (\ref{eq:sympcond}) has attracted much attention in view of the importance of Hamiltonian systems in the applications. On the other hand, it is fair to say that quadratic conservation has been to some extent played down in the geometric integration literature. For this reason, while schemes satisfying (\ref{eq:sympcond}) could have been called conservative, the following terminology is standard:

\begin{definition}The RK scheme (\ref{eq:rkabc}) is called {\em symplectic} (or canonical) if (\ref{eq:sympcond}) holds.
\end{definition}

 Our focus in this paper is on symplectic schemes in as far as they conserve quadratic invariants, as these are actually crucial in several applications. The discussion of any possible benefits derived from the symplectic character of the map $y_n\mapsto y_{n+1}$, including the existence of modified Hamiltonian systems, are out of our scope here. The paper \cite{vilmart} is, in this sense, complementary to the present work.

It was proved in \cite{abia} that the relations (\ref{eq:sympcond}) act as simplifying assumptions {\em vis-\`{a}-vis} the order conditions: once these relations are imposed, the order conditions corresponding to the different elementary differentials/rooted trees are no longer independent. For instance, it is a simple exercise to show that, when (\ref{eq:sympcond}) holds, the second order condition in (\ref{eq:ordcond1}) is a consequence of the first and therefore symplectic RK schemes of order $\geq 1$ automatically possess order $\geq 2$. Similarly the last order condition in (\ref{eq:ordcond1}) is a consequence of the first three.
In this way, for a general RK methods to have order $\geq 3$ for autonomous problems, there are 4 order conditions; for symplectic methods the number is only 2. For a symplectic RK method to have order $\geq \rho$ for autonomous problems there is an order condition for each so-called non-superfluous free tree with $\leq \rho$ vertices.

There are many symplectic RK methods \cite{ssc} including the Gauss methods (of maximal order $2s$ and positive weights) as first shown in \cite{bit}; however no symplectic RK scheme is explicit. The simplest Gauss method ($s=1$) is the familiar implicit midpoint rule.
\subsection{Partitioned Runge-Kutta schemes}
In some applications the components of the vector $y$ in (\ref{eq:ode}) appear partitioned into two blocks: $y = [q\T,p\T]\T$, $q\in\mathbb{R}^{D-d}$, $p\in \mathbb{R}^{d}$. Hamiltonian problems, where $d = D/2$, provide an example, as we have just seen. In those cases it may make sense to use a set of coefficients (\ref{eq:rkabc}) for the integration of the block $q$ and a second set
\begin{equation}\label{eq:rkABC}
A_{ij},\quad i,j = 1,\dots,s, \qquad B_i,\: C_i,\quad i=1,\dots, s,
\end{equation}
for the integration  of the block $p$. (There is no loss of generality in assuming that the number of stages $s$  in (\ref{eq:rkABC}) coincides with that in (\ref{eq:rkabc}): see \cite{ssc} Remark 3.2.)
The overall method is called a PRK scheme.
A more precise description follows.

 Denote by $F=[f\T,g\T]\T$, $f\in\mathbb{R}^{D-d}$, $g\in \mathbb{R}^{d}$ the  partitioning of $F$ induced by the partitioning  $[q\T,p\T]\T$ of $y$, so that (\ref{eq:ode}) reads
 \begin{equation}\label{eq:pode}
 \frac{d}{dt} q = f(q,p,t),\qquad \frac{d}{dt} p = g(q,p,t);
 \end{equation}
 then
the equations for the step $n\rightarrow n+1$ of the PRK method (\ref{eq:rkabc}), (\ref{eq:rkABC}) are
\begin{equation}\label{eq:prkstep1}
q_{n+1} = q_n +h_n \sum_{i=1}^s b_i k_{n,i},\quad
p_{n+1} = p_n +h_n \sum_{i=1}^s B_i \ell_{n,i}, \quad n = 0,\dots, N-1,
\end{equation}
where
\begin{equation}\label{eq:new2}
k_{n,i}= f(Q_{n,i}, P_{n,i}, t_n+c_ih_n),\qquad
\ell_{n,i}=g(Q_{n,i}, P_{n,i}, t_n+C_ih_n),
\end{equation}
and the internal stages  $Q_{n,i}$,  $P_{n,i}$, $i=1,\dots, s$, are defined by the relations
\begin{equation}\label{eq:prkstages1}
Q_{n,i} = q_n +h_n \sum_{i=1}^s a_{ij} k_{n,j}, \qquad
P_{n,i} = p_n +h_n \sum_{j=1}^s A_{ij} \ell_{n,j}.
\end{equation}

 PRK methods are not a mathematical nicety: the Verlet algorithm, the method of choice in molecular dynamics \cite{schlick} is one of them. In its so-called velocity form, the algorithm is written in the molecular dynamics literature as
 (it is a simple matter to rewrite the algorithm in the format (\ref{eq:prkstep1})--(\ref{eq:prkstages1})):
\begin{eqnarray*}
p_{n+1/2}& =& p_n + \frac{h_n}{2} g(q_n,t_n),\\ q_{n+1} &=& q_n + h_nM^{-1}p_{n+1/2},\\ p_{n+1} & =& p_{n+1/2} + \frac{h_n}{2} g(q_{n+1},t_{n+1}).
\end{eqnarray*}
Here the vectors $p$, $q$ and $g$ contain respectively the momenta, positions and forces and $M$ is the diagonal matrix of the masses. Note the way the $q$ and $p$ variables are advanced in different ways.

Clearly an RK scheme may be regarded as a particular instance of a PRK method where the two sets (\ref{eq:rkabc}), (\ref{eq:rkABC}) happen to coincide. For PRK methods to possess order $\geq \rho$ for autonomous problems, there is an order condition associated with each bicolour rooted tree with $\rho$ or less vertices (see e.g. \cite[Chapter III]{hlw}). For order $\geq 2$ the order conditions are:
\begin{eqnarray}\label{eq:ordcond2}
&\sum_{i}b_i = 1,\quad \sum_{i} B_i = 1,&\\ &\sum_{ij} b_ia_{ij} =\frac{1}{2},\quad \sum_{ij} b_iA_{ij} =\frac{1}{2},\quad\sum_{ij} B_ia_{ij} =\frac{1}{2},\quad\sum_{ij} B_iA_{ij} =\frac{1}{2};&\label{eq:ordcond3}
\end{eqnarray}
they correspond to the elementary differentials $f$, $g$, $(\partial_x f)f$, $(\partial_x f)g$, $(\partial_x g)f$, $(\partial_x g)g$ respectively.
It will be important later to note that, if the PRK (\ref{eq:rkabc}), (\ref{eq:rkABC})
has order $\rho$, then the RK scheme with coefficients (\ref{eq:rkabc}) and the RK scheme with coefficients (\ref{eq:rkABC}) have both order $\rho$. The converse  is not true: if
(\ref{eq:rkabc}) and (\ref{eq:rkABC})  are the coefficients of two RK schemes of order $\rho$, then the combined PRK scheme may have order $<\rho$. This is plain  in (\ref{eq:ordcond3}), where the second and third relations are necessary for  the PRK to have order $\geq 2$ but are obviously not required for (\ref{eq:rkabc}) and (\ref{eq:rkABC}) to be the coefficients of two different RK schemes of order $\geq 2$.

For PRK methods, the result corresponding to Theorem \ref{th:cooper} is (cf.  \cite[Chapter IV, Theorem 2.4]{hlw},  where only the autonomous case is envisaged):
\begin{theorem}\label{th:cooper2} Assume  that $S(\cdot,\cdot)$ is a real-valued bilinear map in $\mathbb{R}^d\times \mathbb{R}^{D-d}$ such that, for each $t_0$ and $A$, the solution $y(t) = [q(t)\T,p(t)\T]\T$ of (\ref{eq:ic}), (\ref{eq:pode}), satisfies $$\frac{d}{dt} S(q(t),p(t)) \equiv 0.$$ The relations
\begin{equation}\label{eq:sympcond2}
 b_i = B_i, \quad i=1,\dots, s,\quad b_iA_{ij}+B_ja_{ji} -b_iB_j=0, \quad i,j = 1,\dots, s,
\end{equation}
and
\begin{equation}\label{eq:sympcond2c}
\quad c_i=C_i, \quad i=1,\dots, s,
\end{equation}
 guarantee that, for each PRK trajectory  satisfying (\ref{eq:prkstep1})--(\ref{eq:prkstages1}), $S(q_{n},p_{n})$ is independent of $n$.
\end{theorem}

As in the case of RK methods, the condition in the theorem is necessary for conservation to hold for all $S$ and all partitioned differential systems, see \cite[Chapter VI, Theorems 7.6, 7.10]{hlw}.
In the particular case of autonomous problems the abscissas play no role. Thus, to achieve conservation, it is not necessary to impose the condition (\ref{eq:sympcond2c}) whenever $f$ and $g$ are independent of $t$.
Note that the theorem only applies to a quadratic function of the form $S(q,p)$ which is not the most general possible; for instance the inner product $q\T q$ is not included in that format.

Before proving the theorem we present a simple algebraic auxiliary result that will  be used repeatedly later in other contexts.
\begin{lemma} \label{lemma} Let $q_n$, $p_n$, $Q_i$, $P_i$, $k_{n,i}$, $\ell_{n,i}$ be arbitrary vectors satisfying (\ref{eq:prkstep1}) and (\ref{eq:prkstages1}). If $S$ is bilinear and (\ref{eq:sympcond2}) holds, then
\begin{equation}\label{eq:lemma}
S(q_{n+1},p_{n+1}) -S(q_n,p_n) = h_n\sum_i b_i \Big( S(k_{n,i},P_{n,i})+S(Q_{n,i},\ell_{n,i}) \Big).
\end{equation}
\end{lemma}

\begin{proof}  Since  $S$ is  bilinear, we may write from (\ref{eq:prkstep1})
\begin{eqnarray*}
S(q_{n+1},p_{n+1})-S(q_n,p_n)& = & h_n \sum_i b_i S(k_{n,i}, p_n)+ h_n \sum_j B_j S(q_n, \ell_{n,j})\\&& \qquad\qquad{} + h_n^2 \sum_{ij} b_i B_j S(k_{n,i},\ell_{n,j}).
\end{eqnarray*}
Now use (\ref{eq:prkstages1}) to eliminate $q_n$ and $p_n$ from the right-hand side:
\begin{eqnarray*}
S(q_{n+1},p_{n+1}) -S(q_n,p_n)& = & h_n \sum_i b_i S(k_{n,i}, P_{n,i} - h_n \sum_j A_{ij} \ell_{n,j})\\&&\quad\quad\quad + h_n \sum_j B_j S(Q_{n,j}-\sum_i a_{ji} k_{n,i}, \ell_{n,j})\\&&\quad\quad\quad{} + h_n^2 \sum_{ij} b_i B_j S(k_{n,i},\ell_{n,j}).
\end{eqnarray*}
In view of the bilinearity and (\ref{eq:sympcond2}), the proof is complete.
\end{proof}
\medskip

\noindent{\bf Proof of the theorem:}
Conservation of $S$ implies that $$S(f(q,p,t),p)+S(q,g(q,p,t))\equiv 0,$$ because, along each solution $q(t)$, $p(t)$,
$$
 S\big(\frac{d}{dt} q(t),p(t)\big) + S\big(q(t),\frac{d}{dt} p(t)\big)=\frac{d}{dt}S(q(t),p(t))   = 0.
$$
Therefore (\ref{eq:new2}) and (\ref{eq:sympcond2c}) entail that
the right-hand side of (\ref{eq:lemma}) vanishes.
$\Box$
\medskip

For the preservation of the symplectic structure, the result (derived in \cite{suris2} and \cite{abia2} independently) is:
\begin{theorem}\label{th:lsss2}
Assume that the system (\ref{eq:pode}) is Hamiltonian. The relations (\ref{eq:sympcond2})--(\ref{eq:sympcond2c}) guarantee that the mapping $(q_n,p_n) \mapsto (q_{n+1},p_{n+1})$  defined in (\ref{eq:prkstep1})--(\ref{eq:prkstages1}) is symplectic.
\end{theorem}

The conditions (\ref{eq:sympcond2})--(\ref{eq:sympcond2c})  are essentially necessary for symplecticness \cite{ssc} and hence the following definition:
\begin{definition}The PRK scheme (\ref{eq:rkabc}), (\ref{eq:rkABC}) is called {\em symplectic} if (\ref{eq:sympcond2})--(\ref{eq:sympcond2c}) hold.
\end{definition}

If the PRK is symplectic, there is a reduction in the number of independent order conditions; the classes of equivalent order conditions were first described by Hairer \cite{hairer}. An alternative treatment (see \cite{ander}) based on so-called H-trees was given by Murua in his 1995 thesis, cf.\ \cite{numer}. For instance, for a symplectic PRK method to have order $\geq 4$ it is necessary to impose 13 order conditions: for general PRK methods that number is 36.

\section{Variational systems and their adjoints}
\label{s:adjoint}

We now explore the role of symplectic RK schemes when integrating adjoint variational systems. A comprehensive discussion of the use of adjoints to determine sensitivities is not within our scope here. The paper \cite{giles} provides a general introduction, together with applications to aerodynamics. Applications of adjoints to atmospheric models are discussed in \cite{sandu}. Of course the idea of an adjoint problem is  not restricted to differential equations; see \cite{cacuci} for an early paper describing a very general framework.

\subsection{The continuous problem: quadratic conservation}

 We now present the math\-ematical found\-ations of the remainder of the paper. Consider a $d$-dimensional differential system
\begin{equation}\label{eq:odex}
\frac{d}{dt}x = f(x,t)
\end{equation}
  and denote by $\alpha\in\mathbb{R}^d$  the corresponding initial value and by $\bar x(t)$ the solution that arises from the perturbed initial condition $\bar x(t_0) =\alpha+\eta$.
Linearisation of (\ref{eq:odex}) around $x(t)$ shows that, as $|\eta|\rightarrow 0$, $\bar x(t) = x(t) + \delta(t) +o(|\eta|)$, where $\delta$ solves the  (linear) {\em variational system} (see e.g.\ \cite{hnw} Section I.14)
\begin{equation}\label{eq:var}
\frac{d}{dt} \delta = \partial_x f(x(t),t)\, \delta,
\end{equation}
($\partial_x f$ is the Jacobian matrix of $f$ with respect to $x$). Thus, when $x(t)$ is known, solving for $\delta(t_0+T)$ the initial-value problem given by (\ref{eq:var}) and $\delta(t_0)=\eta$ yields an estimate for the change in solution $\bar x(t)-x(t)$; see a simple example in Fig.~\ref{fig:lotka}.

The {\em adjoint} system of (\ref{eq:var}) is given by
\begin{equation}\label{eq:adj}
\frac{d}{dt} \lambda = - \partial_x f(x(t),t)\T\, \lambda.
\end{equation}
(To avoid confusion, variables in this paper are always {\em column vectors}; from a mathematical point of view it would have been better to write sensitivities, Lagrange multipliers and momenta as row vectors, as they belong to the dual space of the space of states.)
The right-hand side in  (\ref{eq:adj}) has been chosen in such a way that  the following proposition is valid. More precisely, it is best to think that {\em the adjoint is the system  for which the conservation property (\ref{eq:prop})  below holds}.
\begin{proposition}\label{prop}
For each $x$, $\delta$, $\lambda\in\mathbb{R}^d$ and real $t$:
\begin{equation*}
 \big(-\partial_x f(x,t)\T\, \lambda\big)\T\delta+ \lambda\T \partial_x f(x,t)\delta = 0.
\end{equation*}

 Therefore if $\delta(t)$ and  $\lambda(t)$ are arbitrary solutions of (\ref{eq:var}),  (\ref{eq:adj}) respectively, then
\begin{equation}\label{eq:ddd}
\frac{d}{dt} \lambda(t)\T\delta(t) = \Big(\frac{d}{dt} \lambda(t)\Big)\T\delta(t)+\lambda(t)\T\Big(\frac{d}{dt} \delta(t)\Big) \equiv 0
\end{equation}
and accordingly
\begin{equation}\label{eq:prop}
\lambda(t_0+T)\T\delta(t_0+T)= \lambda(t_0)\T\delta(t_0).
\end{equation}
\end{proposition}

Why is the adjoint system useful?
 Regard $\eta$ as a parameter and assume that we are interested in finding $\omega\T\delta(t_0+T)$ for fixed $\omega\in\mathbb{R}^d$, i.e.\ in estimating, at the final time $t_0+T$, the change along  the direction of $\omega$ of the solution of (\ref{eq:odex}) induced by the initial perturbation $\alpha \mapsto \alpha+\eta$. (For instance choosing $\omega$ equal to  the $r$-th co-ordinate vector would correspond to estimating the change in the $r$-th component of the solution.) When $x(t)$ is known, we solve (\ref{eq:adj}) with the {\em final} condition
$\lambda(t_0+T)=\omega$ and note that the quantity we seek coincides with $\lambda(t_0)\T\eta$ because, from the proposition,
$$
\omega\T\delta(t_0+T) = \lambda(t_0+T)\T\delta(t_0+T) = \lambda(t_0)\T\delta(t_0) =\lambda(t_0)\T\eta.
$$
The advantage of this procedure is that, as $\eta$ varies, the computation of $\lambda(t_0)\T\eta$ requires only {\em one} integration of (\ref{eq:adj}); the computation of $\omega\T\delta(t_0+T)$ via (\ref{eq:var}) would need a fresh integration for each new choice of $\eta$  (see Fig.~\ref{fig:lotka}).

As an application, consider the task of computing the  gradient, $\nabla_\alpha{\cal C}( x(t_0+T) )$,  of  a real-valued  function  $\cal C$ with respect to the initial data $\alpha$. We set $\omega =\nabla_x{\cal C}(x(t_0+T))$ in the preceding construction and successively let the $r$-th coordinate vector, $r=1,\dots,d$, play the  role of $\eta$ to conclude that
the gradient sought has the value $\lambda(t_0)$ where $\lambda(t)$ is the solution of the adjoint system with final condition $\lambda(t_0+T) = \nabla_x{\cal C}(x(t_0+T))$. Only one integration is required to find $d$ derivatives $\partial/\partial \alpha^r$. The adjoint system (\ref{eq:adj}) \lq pulls back\rq\ gradients with respect to $x(t_0+T)$ into gradients with respect to $x(t_0)$.

\begin{figure}[h]
\vskip -6 truecm
\hskip -0.75 truecm
\includegraphics[scale=.70]{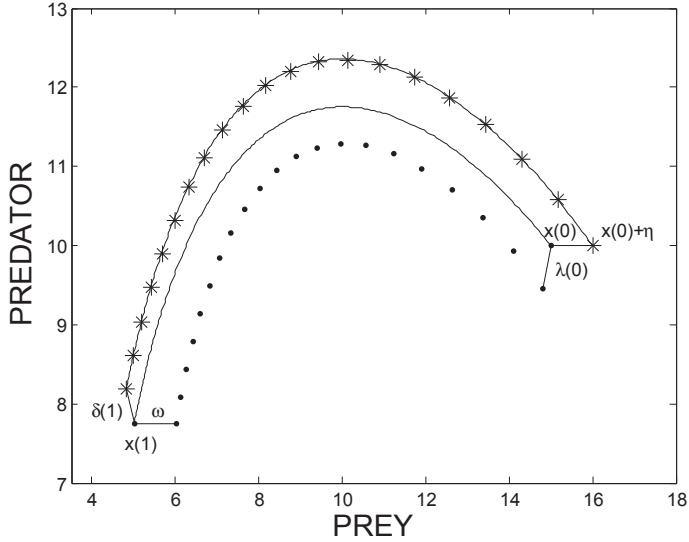}%
\vskip -6 truecm
\caption{\small Two-species Lotka-Volterra system $dx^1/dt = x^1-0.2x^1x^2$, $dx^2/dt = -2x^2+0.2x^1x^2$ (superscripts indicate components of vectors); $x^1$ and $x^2$ represent, in suitable units, numbers of preys and predators respectively. The solid lines give, for $0\leq t\leq 1$, the unperturbed solution $x(t)$ with initial condition $ x(0) = (15,10)$ and a perturbed solution $\bar x(t)$ with
$\bar x(0) =x(0)+\eta =(16,10)$: an increase in the number of preys at $t=0$ leads at $t=1$ to a decrease in the number of preys  and to an increase in the number of predators. The stars are the points $x(t)+\delta(t)$, $ t = 0, 0.05, 0.10, \dots$, where $\delta$ solves the variational system; they almost coincide with the corresponding values of
the perturbed solution
$\bar x(t)$. In particular, the change in the number of preys, $\bar x^1(1)-x^1(1)$, is very well approximated by $\delta^1(1)= -0.1786\dots$, i.e. by the inner product $\omega\T \delta(1)$, where $\omega$ denotes the first co-ordinate vector $(1,0)=\nabla x^1$. The variational equations move $\eta =\delta(0)$ forward to $\delta(1)$. The dots show how the adjoint equations move $\omega = \lambda(1)$ backward to yield $\lambda(0) =\nabla_{x(0)} x^1(1)$, the gradient of $x^1$ as a function of $x(0)$. The inner product  $\omega\T \delta(1)$ exactly coincides with $\lambda(0)\T \eta$. In a Lotka-Volterra system with $d$ species, a single integration of the adjoint system is necessary to find the $d$-dimensional gradient of $x^1(1)$ as a function of $x(0)$. }
\label{fig:lotka}
\end{figure}

\subsection{The continuous problem: Lagrange multipliers}
\label{ss:lagrange}
We shall also need  an alternative derivation of the recipe  $\nabla_\alpha{\cal C}( x(t_0+T) )=\lambda(t_0)$ just found. Since the use of Lagrange multipliers (see e.g. \cite[Section 2.5]{giles}) in this connection (as distinct from their use in minimisation) may not be known to some readers, we give full details. Define the
Lagrangian functional ${\mathcal L}= {\mathcal L}(\hat \alpha,\hat x,\hat \lambda_0,\hat \lambda)$
$$
{\mathcal L} ={\cal C}(\hat x(t_0+T))-\hat\lambda_0\T\big(\hat x(t_0)-\hat\alpha\big)-\int_{t_0}^{t_0+T} \hat\lambda(t)\T \Big(\frac{d}{dt} \hat x(t) - f(\hat x(t),t)\Big)\,dt,
$$
where, $\hat \alpha$, $\hat \lambda_0$ are arbitrary vectors, $\hat x$, $\hat \lambda$ arbitrary functions. A key point here is that, whenever $\hat x$ is a solution of (\ref{eq:odex}) and $\hat x(t_0) = \hat \alpha$, the value of ${\cal L}(\hat\alpha,\hat x,\hat \lambda_0,\hat \lambda)$ coincides with ${\cal C}(\hat x(t_0+T))$.

If $\eta$ and $\delta$ are the variations of $\hat\alpha$ and $\hat x$ respectively, the variation $\delta \mathcal L$ of the functional is
\begin{eqnarray*}
\delta {\mathcal L}&=&\nabla_x {\cal C}(\hat x(t_0+T))\T \delta(t_0+T)-\hat\lambda_0\T \big(\delta(t_0)-\eta\big)\\
&& \qquad\qquad-\int_{t_0}^{t_0+T} \hat\lambda(t)\T \Big(\frac{d}{dt} \delta(t) - \partial_x f(\hat x(t),t)\delta(t)\Big)\,dt,
\end{eqnarray*}
so that, after integration by parts,
\begin{eqnarray*}
\delta{\mathcal L}&=&\big(\nabla_x {\cal C}(\hat x(t_0+T))-\hat\lambda(t_0+T)\big)\T \delta(t_0+T) +\hat\lambda(t_0)\T \eta \\
&&\qquad  {}+\big(\hat\lambda(t_0)-\hat\lambda_0\big)\T \delta(t_0)\\
&&\qquad{}+ \int_{t_0}^{t_0+T} \Big( \frac{d}{dt}\hat\lambda(t)^T\delta(t)+\hat\lambda(t)^T\partial_x f(\hat x(t),t)\delta(t)\Big)\,dt.
\end{eqnarray*}
We now make  choices $\lambda_0$, $\lambda$ (depending on $\hat\alpha$ and $\hat x$) for the (so far arbitrary) multipliers $\hat\lambda_0$, $\hat\lambda$.  We define $\lambda$ as the solution of the equation (\ref{eq:adj}) (with $\hat x(t)$ in lieu of $x(t)$) subject to the final condition $\lambda(t_0+T)=\nabla_x {\cal C}(\hat x(t_0+T))$ and  set
$\lambda_0=\lambda(t_0)$.
These choices ensure that, at  $\hat\alpha$, $\hat x$, the {\em intermediate} variation $\delta(t)$ does not contribute to $\delta \mathcal L$; we then have (at $\hat \alpha$, $\hat x$) $\delta {\mathcal L}= \lambda(t_0)\T \eta$ or, in other words, $\lambda(t_0)$ is the gradient of $\mathcal L$ as a function of $\hat\alpha$. Since, as pointed out above, if $\hat x$ solves (\ref{eq:odex}) and $\hat x(t_0) = \hat \alpha$, then ${\cal L}(\hat\alpha,\hat x,\hat \lambda_0,\hat \lambda)={\cal C}(\hat x(t_0+T))$, we conclude that
$\lambda(t_0) =\nabla_\alpha{\cal C}( x(t_0+T) )$ as we wished to prove. The original system (\ref{eq:odex}) and the initial condition may also be retrieved from the Lagrangian by making zero the variations with respect to  $\hat \lambda$ and $\hat \lambda_0$ respectively.

The same approach may also be used  if we wish to make things more involved and introduce the velocity $(d/dt)\hat x = \hat k$  as a new argument in the Lagrangian. To simplify the notation we shall hereafter
drop all hats, so that the same symbols $\alpha$, $x$, \dots will be used for the arbitrary arguments of the Lagrangian (that previously were written as $\alpha$, $x$, \dots) and for the corresponding values at the solution sought.
When the velocity is considered as a new argument, the Lagrangian becomes
\begin{eqnarray}\nonumber
{\mathcal L} &=&{\cal C}(x(t_0+T))-\lambda_0\T\big(x(t_0)-\alpha\big)\\&&\nonumber\qquad{}
-\int_{t_0}^{t_0+T} \lambda(t)\T \Big(\frac{d}{dt} x(t) - k(t)\Big)\,dt\\&&\qquad{}-\int_{t_0}^{t_0+T} \Lambda(t)\T \Big(k(t) - f(x(t),t)\Big)\,dt\label{eq:friday}.
\end{eqnarray}
Taking variations and choosing the multipliers to cancel the undesired contributions to $\delta \mathcal L$, leads to the relations $\lambda(t_0) =\nabla_\alpha{\cal C}( x(t_0+T) )$, $\lambda(t_0+T)=\nabla_x {\cal C}(x(t_0+T))$, $\lambda_0=\lambda(t_0)$
 found above and, additionally, to $\Lambda(t) \equiv \lambda(t)$ (as expected).

\subsection{The discrete problem: RK integration}

Let us suppose that (\ref{eq:odex}) has been discretised by means of the RK scheme (\ref{eq:rkabc}) to get, $n = 0,\dots,N-1$,
\begin{eqnarray}
\label{eq:rkstepx}
x_{n+1} &=& x_n +h_n \sum_{i=1}^s b_i {k}_{n,i},
\\
\label{eq:new3}
{k}_{n,i}& =&f(X_{n,i},t_n+c_ih_n),  \quad i=1,\dots, s,\\
\label{eq:rkstagesx}
X_{n,i} &=& x_n + h_n \sum_{j=1}^s a_{ij} {k}_{n,j}, \quad i=1,\dots, s,
\end{eqnarray}
and that, in analogy with the preceding material, we wish to estimate the impact on $x_N$ of a perturbation of the initial condition $x_0=\alpha$. Linearisation of the RK equations (\ref{eq:rkstepx})--(\ref{eq:rkstagesx}) around $x_n$, $X_{n,i}$ shows that the perturbed RK solution
$\bar x_n$, $n=0,\dots, N$, satisfies $\bar x_n= x_n+\delta_n+o(|\eta|)$ with
\begin{eqnarray}\label{eq:rkstepvar}
\delta_{n+1} &=& \delta_n +h_n \sum_{i=1}^s b_i  d_{n,i}, \\
\label{eq:new4}
 d_{n,i}&=& \partial_xf(X_{n,i},t_n+c_ih_n)\Delta_{n,i}, \quad i=1,\dots, s,\\
\label{eq:rkstagesvar}
\Delta_{n,i} &=& \delta_n + h_n \sum_{j=1}^s a_{ij} d_{n,j}, \quad i=1,\dots, s
\end{eqnarray}
(the vectors $d_{n,i}$ and $\Delta_{n,i}$ are the variations in the slopes ${k}_{n,i}$ and stages $X_{n,i}$ respectively).

On the other hand, if we regard the given differential equations (\ref{eq:odex}) together with the variational equations (\ref{eq:var}) as a $2d$-dimensional system for the vector $y=[x\T,\delta\T]\T$ and apply the RK scheme as in (\ref{eq:rkstep})--(\ref{eq:rkstages}), we also arrive at
(\ref{eq:rkstepx})--(\ref{eq:rkstagesvar}).  We have thus proved, as in, say, \cite[Chapter VI, Lemma 4.1]{hlw}:
\begin{theorem}\label{th:variational}The process of RK discretisation commutes with forming variational equations: the RK discretisation of the continuous variational equations (\ref{eq:odex})--(\ref{eq:var}) yields  the variational equations (\ref{eq:rkstepx})--(\ref{eq:rkstagesvar}) for the RK discretisation.
\end{theorem}

The situation for the adjoint equations is not quite as neat (cf. \cite{sirkes}). In order to find the discrete sensitivity $\omega\T\delta_N$ we would like to numerically integrate (\ref{eq:adj}) with final condition $\lambda_N =\omega$ in such a way that (cf.\ (\ref{eq:prop}))
\begin{equation}\label{eq:disprop}
\lambda_N\T\delta_N= \lambda_0\T\delta_0.
\end{equation}
Although in actual computation the approximations $\lambda_n$ are to be found without using the equations (\ref{eq:rkstepvar})--(\ref{eq:rkstagesvar}) for $\delta_n$ (this is the whole point behind the use of adjoints), let us consider for a moment the $3d$-dimensional system (\ref{eq:odex})--(\ref{eq:adj}) for the extended vector $y=[x\T,\delta\T,\lambda\T]\T$. Then the condition (\ref{eq:disprop}) demands that we integrate this large system in such a way as to {\em exactly} preserve the invariant $I(y(t),y(t)) = \lambda(t)\T\delta(t)$  in (\ref{eq:ddd}). According to Theorem \ref{th:cooper}, we may achieve this goal by using the RK scheme (\ref{eq:rkabc}) {\em provided that  it is  symplectic}. This results in the
relations (\ref{eq:rkstepx})--(\ref{eq:rkstagesvar}) in tandem with ($n = 0,\dots,N-1$):
\begin{eqnarray}\label{eq:rksteplambda}
\lambda_{n+1} &=& \lambda_n +h_n \sum_{i=1}^s b_i \ell_{n,i}, \\
\label{eq:new5}
\ell_{n,i} &=& -\partial_xf(X_{n,i}, t_n+c_ih_n)\T \Lambda_{n,i}, \quad i=1,\dots, s,
\\
\label{eq:rkstageslambda}
\Lambda_{n,i} &=& \lambda_n +h_n \sum_{j=1}^s a_{ij}\ell_{n,j}, \quad i=1,\dots, s.
\end{eqnarray}

Let us summarise the preceding discussion:

\begin{theorem}\label{th:adjoint}
Assume that the $3d$-dimensional system (\ref{eq:odex})--(\ref{eq:adj}) is discretised by a {\em symplectic} RK scheme (\ref{eq:rkabc}). Then for any RK solution
(\ref{eq:disprop}) holds. In particular, for the RK solution specified by the initial condition $x_0 = \alpha$, $\delta_0=\eta$ together with the final condition $\lambda_N=\omega$,
$$
\omega\T\delta_N = \lambda_0\T\eta.
$$
\end{theorem}

For a non-symplectic RK scheme of order $\rho$, $\omega\T\delta_N$ and  $\lambda_0\T\eta$ are approximations of order $\rho$  to their continuous counterparts $\omega\T\delta(t_0+T)$
and $\lambda(t_0)\T \eta$ respectively and  therefore $\lambda_0\T\eta$ will be a $\mathcal{O}(h^\rho)$ approximation to the true sensitivity $\omega\T\delta_N$ of the discrete solution. See the example in Table~\ref{table:table} where the Euler integrator was chosen so as to have large errors and see clearly the difference between $\omega\T\delta_N$ and  $\lambda_0\T\eta$.

\begin{table}
\begin{center}
\begin{tabular}{ccccc}
$h$ & $\lambda_0\T\eta$ & $\omega\T\delta_N$ & $\lambda_0\T\eta-\lambda(0)\T\eta$&   $\omega\T\delta_N-\omega\T \delta(1)$\\\hline
$0.100$ & $-0.1070$ & $-0.2497$ & $0.0717$ & $-0.0710$\\
$0.050$ & $-0.1401$ & $-0.2135$ & $0.0385$ & $-0.0348$\\
$0.025$ & $-0.1588$ & $-0.1959$ & $0.0199$ & $-0.0172$
\end{tabular}
\end{center}
\caption{\small Euler integration on a uniform grid of the $x$, $\delta$, $\lambda$ equations for the Lotka-Volterra problem in Fig.~\ref{fig:lotka}. The lack of symplecticness of the integrator results in $\lambda_0\T\eta$ being different from $\omega\T\delta_N$: the discretisation of the adjoint equations does not provide the adjoint of the discretisation. The convergence of the integrator implies that, as the grid is refined, $\lambda_0\T\eta$ and $\omega\T\delta_N$ are $\mathcal{O}(h)$ away from their common limit $\lambda(0)\T\eta=\omega\T\delta(1)\approx -0.1786$, as borne out by the last two columns. When, alternatively, the $\lambda$ equations are integrated with the Radau method (\ref{eq:radau})
the numerical results for $\lambda_0\T\eta$ coincide with those displayed in the third column of the table.
}
\label{table:table}
\end{table}

In practice, the variational equations (\ref{eq:var}) do {\em not need to be integrated}.  We successively find $x_0$, $x_1$, \dots, $x_N$ via (\ref{eq:rkstepx})--(\ref{eq:rkstagesx}) and, once these are available, we set $\lambda_N=\omega$,  and compute $\lambda_{N-1}$, \dots, $\lambda_0$ from (\ref{eq:rksteplambda})--(\ref{eq:rkstageslambda}) taken in the order $n =N-1, N-2, \dots,0$. For this reason, it may be advisable to rewrite (\ref{eq:rksteplambda})--(\ref{eq:rkstageslambda}) in the following \lq reflected\rq\ form (see Section~\ref{s:scherer}) that emphasises that the approximation $\lambda_n$ at $t_n$ is to be found from the approximation $\lambda_{n+1}$ at $t_{n+1}$:
\begin{eqnarray}\label{eq:rksteplambdabis}
\lambda_{n} &=& \lambda_{n+1} +(-h_n) \sum_{i=1}^s b_i \ell_{n,i}, \\
\label{eq:new5bis}
\ell_{n,i} &=& -\partial_xf(X_{n,i}, t_{n+1}+(1-c_i)(-h_n))\T \Lambda_{n,i}, \quad i=1,\dots, s,
\\
\label{eq:rkstageslambdabis}
\Lambda_{n,i} &=& \lambda_{n+1} +(- h_n) \sum_{j=1}^s (b_j-a_{ij})\ell_{n,j}, \quad i=1,\dots, s.
\end{eqnarray}

 In analogy to the continuous case, for a symplectic RK discretisation, $\nabla_\alpha{\cal C}(x_N)$ may be computed by finding $\lambda_0$ from the recursion (\ref{eq:rksteplambda})--(\ref{eq:rkstageslambda})  (or (\ref{eq:rksteplambdabis})--(\ref{eq:rkstageslambdabis})) with $\lambda_N = \nabla_x {\cal C}(x_N)$.

\subsection{The discrete problem: PRK integration}

Theorem \ref{th:adjoint} may be generalised easily with the help of Theorem \ref{th:cooper2}. Hereafter it is understood that when using the PRK scheme the $x$, $\delta$ equations are integrated with the  set of coefficients
(\ref{eq:rkabc}) (so that the $\delta_n$ are exactly the variations in $x_n$) and the $\lambda$ equations with the set of coefficients (\ref{eq:rkABC}). In other words, the system is partitioned as $q = [x\T,\delta\T]\T$, $p=\lambda$.\footnote{A variation on this theme is presented in \cite[Section 6]{sina} in the context of optimal control problem. There the $x$ equations are themselves partitioned and integrated by means of a symplectic PRK.} This approach leads to (\ref{eq:rkstepx})--(\ref{eq:rkstagesvar}) supplemented by the relations obtained by replacing the lower case coefficients $a_{ij}$, $b_i$, $c_i$ in (\ref{eq:rksteplambda})--(\ref{eq:rkstageslambda}) by their upper case counterparts:
\begin{eqnarray}\label{eq:rksteplambdater}
\lambda_{n+1} &=& \lambda_n +h_n \sum_{i=1}^s B_i \ell_{n,i}, \\
\label{eq:new5ter}
\ell_{n,i} &=& -\partial_xf(X_{n,i}, t_n+C_ih_n)\T \Lambda_{n,i}, \quad i=1,\dots, s,
\\
\label{eq:rkstageslambdater}
\Lambda_{n,i} &=& \lambda_n +h_n \sum_{j=1}^s A_{ij}\ell_{n,j}, \quad i=1,\dots, s.
\end{eqnarray}

The generalisation of Theorem \ref{th:adjoint} is:
\begin{theorem}\label{th:adjointPRK}
Assume that the $3d$-dimensional system (\ref{eq:odex})--(\ref{eq:adj}) is discretised by a {\em symplectic} PRK scheme (\ref{eq:rkabc}), (\ref{eq:rkABC}). Then
(\ref{eq:disprop}) holds for any PRK solution. In particular, for the PRK solution specified by the initial condition $x_0 = \alpha$, $\delta_0=\eta$ together with the final condition $\lambda_N=\omega$,
$$
\omega\T\delta_N = \lambda_0\T\eta.
$$
\end{theorem}

Once more, for a symplectic PRK discretisation, the gradient $\nabla_\alpha{\cal C}(x_N)$ coincides with $\lambda_0$  if $\lambda_N = \nabla_x {\cal C}(x_N)$. For a non-symplectic discretisation of the adjoint equations, $\lambda_0$ is a only an approximation to $\nabla_\alpha{\cal C}(x_N)$. For this reason {\em  non-symplectic PRK discretisations cannot be implied by the direct differentiation procedure} described in Section \ref{ss:navidad}.

How do we compute {\em exactly} (i.e.\ up to round-off) the sensitivity $\omega\T\delta_N$ with the help of the adjoint system when the $x$ integration has been performed with a non-symplectic RK scheme (\ref{eq:rkabc})  and Theorem \ref{th:adjoint} cannot be invoked? Theorem \ref{th:adjointPRK} suggests the way. For simplicity we only look at the case where in (\ref{eq:rkabc}) none of the weights
$b_i$, $i=1,\dots,s$, vanishes (for the general situation see the appendix). From the coefficients in (\ref{eq:rkabc}) we compute a new set
\begin{equation}\label{eq:trick}
 A_{ji} =b_i- b_ia_{ij}/b_j,\quad i,j=1,\dots, s, \quad B_i = b_i,\quad C_i=c_i\quad i= 1,\dots,s.
\end{equation}
 In view of (\ref{eq:sympcond2})--(\ref{eq:sympcond2c}), we now have a PRK scheme for the discretisation of (\ref{eq:odex})--(\ref{eq:adj}) and Theorem \ref{th:adjointPRK} applies. If (\ref{eq:rkabc}) is explicit, the computations required to descend from $\lambda_N$ to $\lambda_0$ are also explicit. Here is the simplest example. Assume that the $x$ equations are integrated with the explicit Euler rule ($s=1$,  $a_{11}= 0$, $b_1=1$, $c_1=0$). With that choice, $X_{n,1} = x_n$ and
$$
x_{n+1} = x_n+h_nf(x_n,t_n).
$$
The trick just described yields  $A_{11}=1$, $B_1=1$, $C_1=0$. Accordingly, the stage $\Lambda_{n,1}$ coincides with  $\lambda_{n+1}$ and using (\ref{eq:prkstep1}) we see that the required $\lambda$ integrator is:
\begin{equation}\label{eq:radau}
\lambda_{n+1} = \lambda_n-h_n \partial_xf(x_n,t_n)\T \lambda_{n+1}.
\end{equation}
Obviously this is {\em not} the explicit Euler rule, because $\lambda$ in the right-hand side appears at time $t_{n+1}$. And, unless the problem is autonomous, it is not the implicit Euler rule either because $t$ is evaluated at the retarded time $t_n$. (For RK enthusiasts only: the coefficients $A_{11}=1$, $B_1=1$, $C_1=0$ correspond to the Radau IA method of one stage introduced by Ehle, \cite[Section IV.5]{hw}.)

In the particular situation where the $x$ integration has been performed by  a symplectic RK method (symplectic RK methods possess non-vanishing weights \cite{ssc}, Section 8.2), the recipe (\ref{eq:trick}) will lead to $A_{ij}=a_{ij}$ and the resulting PRK method will coincide with the original RK method. In the general case, for (\ref{eq:disprop}) to hold, {\em the adjoint equations for $\lambda$ have to be integrated with coefficients  different from those used for the original equations for $x$}.

There are  hidden  difficulties with the use of this recipe. When stability is an issue, as in stiff problems or time-discretisations of partial differential equations, it is necessary to investigate carefully the stability behaviour of the $\lambda$ integration \cite{sirkes}.
On the other hand, and as noted before, the order of accuracy of the overall PRK, $x$, $\lambda$, integrator may be lower than the order of the RK method (\ref{eq:rkabc}) for $x$ we started with. When investigating the order of the overall PRK method we have to take into account that the right-hand side of (\ref{eq:odex})  is independent of $\lambda$ and the right-hand side of (\ref{eq:adj}) is linear in $\lambda$. These features imply that many elementary differentials vanish and that accordingly it is not necessary to impose the order conditions associated with them. Furthermore we have to take into account the reduction in the number of independent order conditions implied by symplecticness.

\subsection{The discrete problem: automatic differentiation}
\label{ss:navidad}

According to the preceding discussion, for any  RK integration of (\ref{eq:odex}) with nonzero weights, it is possible to find the gradient $\nabla_\alpha {\cal C}(x_N)$ by means of an integration of the adjoint equations with the coefficients (\ref{eq:trick}). It is however clear that it is also perfectly possible to compute $\nabla_\alpha {\cal C}(x_N)$ by repeatedly using the chain rule in (\ref{eq:rkstepx})--(\ref{eq:rkstagesx}), something that we shall perform presently. Since $\cal C$ is scalar and $\alpha\in \mathbb{R}^d$, where $d$ is possibly large,  reverse accumulation \cite{griewank}\footnote{Recall that the idea of reverse accumulation is as follows. Imagine an application of the chain rule that leads to a product $J_3J_2J_1$, where $J_3$ is the Jacobian matrix $\partial(z)/\partial(y)$ of the final variables $z$ with respect to some intermediate variables $y$ and similarly $J_2 =\partial(y)/\partial(x)$, $J_1= \partial(x)/\partial(w)$ ($w$ are the independent variables). When the dimension of $z$ is much lower than the dimensions of $x$, $y$ and $w$, computing  the \lq short\rq\ (few rows) matrices $K =J_3J_2$ and  $KJ_1$ (reverse accumulation) is much cheaper than first forming the \lq tall\rq\ (many rows) matrix $L=J_2J_1$ and then $J_3L$ (forward accumulation). The {\em forward} order $J_3(J_2J_1)$ finds successively the Jacobians $J_1=\partial(x)/\partial(w)$,
$J_2J_1 = \partial(y)/\partial(w)$ and $J_3J_2J_1 = \partial(z)/\partial(w)$. In {\em reverse} mode, the intermediate Jacobians are $J_3 = \partial(z)/\partial(y)$, $J_3J_2 = \partial(z)/\partial(x)$, $J_3J_2J_1=\partial(z)/\partial(w)$. The analogy with the $\delta$ and $\lambda$ equations is manifest.} is to be preferred and this may be performed with the help of Lagrange multipliers as in Section \ref{ss:lagrange}.

We shall need the following auxiliary result:
\begin{lemma}\label{l:automatic}Suppose that the mapping $\Omega:\mathbb{R}^{d+d^\prime}\rightarrow \mathbb{R}^{d^\prime}$ is such that the Jacobian matrix $\partial_\gamma \Omega$ is invertible at a point $(\alpha_0,\gamma_0)\in\mathbb{R}^d\times \mathbb{R}^{d^\prime}$,  so that in the neighborhood of $\alpha_0$, the equation $\Omega(\alpha,\gamma)=0$ defines $\gamma$ as a function of $\alpha$. Consider a real-valued function in $\mathbb{R}^d$ of the form
$\psi(\alpha) = \Psi(\alpha,\gamma(\alpha))$, for some  $\Psi:\mathbb{R}^{d+d^\prime}\rightarrow \mathbb{R}$. There exists a unique vector $\lambda_0\in \mathbb{R}^{d^\prime}$ such that
(superscripts denote components):
\begin{eqnarray*}
 \nabla_\alpha \psi |_{\alpha_0} & = &  \nabla_\alpha \Psi |_{(\alpha_0,\gamma_0)}+\sum_{r=1}^{d^\prime}\lambda_0^r  \nabla_\alpha\Omega^r  |_{(\alpha_0,\gamma_0)},\\
0 & = &  \nabla_\gamma \Psi |_{(\alpha_0,\gamma_0)}+\sum_{r=1}^{d^\prime}\lambda_0^r  \nabla_\gamma\Omega^r  |_{(\alpha_0,\gamma_0)}.
\end{eqnarray*}
\end{lemma}
\begin{proof} The second requirement may be rewritten as
\begin{equation}\label{eq:upper}
(\partial_\gamma \Omega)\T\lambda_0 = - \nabla_\gamma \Psi,
\end{equation}
with the matrix and right-hand side evaluated at $\alpha_0$, $\gamma_0$. This is a linear system that uniquely defines $\lambda_0$.
To check that the vector $\lambda_0$ we have just found satisfies the first requirement, we use the chain rule
$$
\partial_\alpha \psi |_\alpha = \partial_\alpha \Psi |_{(\alpha,\gamma(\alpha))} + \partial_\gamma \Psi |_{(\alpha,\gamma(\alpha))} \partial_\alpha \gamma|_\alpha,
$$
differentiate  $\Omega(\alpha,\gamma(\alpha))= 0$ to get
$$
 \partial_\alpha \Omega |_{(\alpha,\gamma(\alpha))} + \partial_\gamma \Omega |_{(\alpha,\gamma(\alpha))}\partial_\alpha \gamma|_\alpha=0,
$$
evaluate at $\alpha_0$, and eliminate $\partial_\alpha \gamma|_{\alpha_0}$.
\end{proof}\medskip

It is useful to rephrase the lemma by introducing the Lagrangian
$$
 {\mathcal L}(\alpha,\gamma,\lambda) = \Psi(\alpha,\gamma)+\lambda^T \Omega(\alpha,\gamma).
$$
so that the relation $\Omega(\alpha_0,\gamma_0) = 0$ and the equation (\ref{eq:upper}) that defines the multiplier are respectively
$$
\nabla_\lambda  {\mathcal L}(\alpha,\gamma,\lambda)|_{(\alpha_0,\gamma_0,\lambda_0)} = 0,\qquad
\nabla_\gamma  {\mathcal L}(\alpha,\gamma,\lambda)|_{(\alpha_0,\gamma_0,\lambda_0)} = 0,
$$
while the gradient we seek is computed as
$$
 \nabla_\alpha \psi |_{\alpha_0}  = \nabla_\alpha  {\mathcal L}(\alpha,\gamma,\lambda)|_{(\alpha_0,\gamma_0,\lambda_0)}.
$$
Note that these developments mimic the material in Section \ref{ss:lagrange}, with $\gamma$ playing the part of $\hat x$, $\gamma_0$ the part of $x$, etc.

In numerical differentiation,  $\psi$ is the function whose gradient is to be evaluated, the components of $\alpha$ are  the independent variables, and the components of $\gamma$ represent intermediate stages towards the computation of $\psi$. (For instance, in the simple case ($d=1$) where $\psi(\alpha) = \alpha \sqrt{1+\alpha\exp(\alpha)\cos(\exp(\alpha))}$, we may set  the constraints $\Omega^1 = \gamma^1-\exp(\alpha) = 0$,
$\Omega^2 = \gamma^2- \cos(\gamma^1)= 0$, $\Omega^3 = \gamma^3-\alpha\gamma^1\gamma^2=0$, $\Omega^4 =\gamma^4-\sqrt{1+\gamma^3} $, $\psi =\alpha\gamma^4$.) The interpretation of the $\gamma^r$ as successive stages implies that, in practice, $\Omega$ will possess a lower triangular structure: $\Omega^r$ will only involve $\gamma^1$,\dots,$\gamma^r$.
The evaluation of $\psi$ successively finds the numerical values of $\gamma^1$,\dots,$\gamma^{d^\prime}$ in a forward fashion. The numerical values of the components $\lambda_0^r$, are then found by {\em backward} substitution in the upper-triangular linear system (\ref{eq:upper}) and finally the lemma yields the required value of the gradient. If $\Psi$ and $\Omega$ have been judiciously chosen, then the mappings $\nabla_\alpha \Psi$, $\nabla_\gamma \Psi$, $\nabla_\alpha \Omega^r$, $\nabla_\alpha \Omega^r$ required to compute the gradient will
  have simple analytic expressions, easily derived by a human or by a computer programme.

We now  apply this technique to find $\nabla_\alpha {\cal C}(x_N)$. In (\ref{eq:rkstepx})--(\ref{eq:rkstagesx}) we let (the components of) $x_n$, $n= 0,...,N$, and ${k}_{n,i}$,  $n=0,\dots,N-1$, $i=1,\dots,s$, play the role of (the components of) $\gamma$ and introduce the Lagrangian
\begin{eqnarray}\nonumber
& {\cal C}(x_N) - \lambda_0\T(x_0-\alpha)-\sum_{n=0}^{N-1} h_n\lambda_{n+1}\T \Big[\frac{1}{h_n}(x_{n+1} - x_n) - \sum_{i=1}^s b_i {k}_{n,i}\Big]\\
& {} - \sum_{n=0}^{N-1} h_n\sum_{i=1}^s b_i \Lambda_{n,i}\T\Big[ k_{n,i} -f(X_{n,i}, t_n+c_ih_n)\Big],\label{eq:hat}
\end{eqnarray}
where we understand that the stage vectors $X_{n,i}$ have been expressed in terms of the $x_n$ and $k_{n,i}$ by means of (\ref{eq:rkstagesx}). Clearly this discrete Lagrangian  is the natural RK approximation to (\ref{eq:friday}).

A straightforward application of Lemma \ref{l:automatic} now directly yields the following result, where  we note that the hypothesis $b_i \neq 0$, $i=1,\dots,s$, is natural because, when, say,  $b_1=0$,  the Lagrangian (\ref{eq:hat}) does not incorporate the constraint $k_{n,1} = f(X_{n,1}, t_n+c_1h_n)$. (The case of zero weights is considered in the appendix.)
\begin{theorem}\label{th:automatica}  Consider the RK equations (\ref{eq:rkstepx})--(\ref{eq:rkstagesx}), with $b_i \neq 0$, $i=1,\dots,s$. The computation of $\nabla_\alpha{\cal C}(x_N)$ based on the use of Lemma \ref{l:automatic} with Lagrangian (\ref{eq:hat}) leads to the relations (\ref{eq:rksteplambdater})--(\ref{eq:rkstageslambdater}), with the coefficients $A_{ij}$, $B_i$, $C_i$ given by (\ref{eq:trick}), together with $\nabla_x{\cal C}(x_N)=\lambda_N$, $\nabla_\alpha{\cal C}(x_N)=\lambda_0$.
\end{theorem}

Note that, in the situation of the theorem, $\lambda_N$, $\lambda_{N-1}$, $\lambda_{N-2}$, \dots successively yield the gradients $\nabla_{x_N}{\cal C}(x_N)$, $\nabla_{x_{N-1}}{\cal C}(x_N)$, $\nabla_{x_{N-2}}{\cal C}(x_N)$, \dots\
It is well known that the reverse mode of differentiation implies an integration of the adjoint equations. The theorem shows additionally that, for an RK computation of $x$, the implied adjoint equation integration is such that the $x$, $\lambda$ system is discretised with a {\em symplectic} PRK method. Recall that we showed in the preceding subsection that nonsymplectic PRK cannot appear in this setting as they do not find exactly  $\nabla_\alpha{\cal C}(x_N)$. In a way the chain rule  provided us with  symplectic integration {\em before the latter was invented}.

A further remark: the use of the chain rule with forward accumulation implies an RK integration of the variational equations (\ref{eq:var}) with the original RK coefficients (\ref{eq:rkabc}). In agreement with a previous discussion, the forward mode is more expensive; each partial derivative $\partial/\partial\alpha^r$, $r = 1,\dots,d$, in the gradient requires a separate integration.

\section{A simple optimal control problem}
\label{sec:control}
We  explore next the role of symplectic methods when integrating the differential equations that arise in some optimal control problems \cite{sontag}, \cite{trelat}, \cite{zabczyk}. In this section we look at the simplest case, where the developments are very similar to those just considered; more general problems are treated in the next.
\subsection{The continuous problem}
Consider now the $d$-dimensional system
\begin{equation}\label{eq:odecontrol}
\frac{d}{dt}x = f(x,u,t),
\end{equation}
where $x$ is the state vector and $u$ a $\nu$-dimensional vector of controls. Our aim is to find functions $x(t)$ and $u(t)$, subject to (\ref{eq:odecontrol}) and the initial condition $x(t_0) = \alpha\in\mathbb{R}^d$, so as to minimise a given cost function
${\cal C}(x(t_0+T))$.

The variational equation is (cf.\ (\ref{eq:var}))
\begin{equation}\label{eq:varcontrol}
\frac{d}{dt} \delta = \partial_x f(x(t),u(t),t)\, \delta+\partial_uf(x(t),u(t),t)\, \zeta,
\end{equation}
where $\partial_u$ is the Jacobian matrix of $f$ with respect to $u$ and $\zeta$ denotes the variation in $u$, see e.g. \cite[Section 2.8]{sontag}, \cite[Section 5.1]{trelat}. Now $\delta(t_0) =0$, as $x(t_0)$ remains nailed down at $\alpha$.

An adjoint system (cf.\ (\ref{eq:adj}))
\begin{equation}\label{eq:controladjoint}
\frac{d}{dt} \lambda = - \partial_x f(x(t),u(t),t)\T\, \lambda,
\end{equation}
and constraints
\begin{equation}\label{eq:constraints}
\partial_u f(x(t),u(t),t)\T \lambda(t) =0,
\end{equation}
are introduced, see e.g.\ \cite[Section 9.2]{sontag}. As  was the case with the adjoint in (\ref{eq:adj}), the actual form of these equations is chosen to ensure the validity of the conservation property (\ref{eq:prop}). More precisely we have the following result:

\begin{proposition}\label{propcontrol}
For each choice of vectors $x$, $u$, $\delta$, $\zeta$, $\lambda$ and real $t$:
\begin{equation}\label{eq:identity2}
\Big( - \partial_x f(x,u,t)\T\, \lambda\Big)\T\delta +
\lambda\T \Big(\partial_xf(x,u,t)\delta+\partial_u f(x,u) \zeta\Big)=0.
\end{equation}
Therefore if $\delta(t)$, $\lambda(t)$, $\zeta(t)$ satisfy (\ref{eq:varcontrol})--(\ref{eq:constraints}), then (\ref{eq:ddd})--(\ref{eq:prop}) hold.
\end{proposition}

The use of the proposition is as follows. We solve the two-point boundary problem given by the states+costates system (\ref{eq:odecontrol}), (\ref{eq:controladjoint})--(\ref{eq:constraints}) with initial/final  conditions
\begin{equation}\label{eq:conditions}
x(t_0) = \alpha,\qquad \lambda(t_0+T) = \nabla {\cal C}(x(t_0+T)).
\end{equation}
 Then, the variation $\delta(t_0+T)$ at the end of the interval is orthogonal to the gradient of the cost since, from (\ref{eq:prop}),
\begin{equation}\label{eq:ander}
\nabla {\cal C}(x(t_0+T))\T \delta(t_0+T) = \lambda(t_0+T)\T\delta(t_0+T) = \lambda(t_0)\T\delta(t_0) =0.
\end{equation}
 This of course means that any  solution $[x(t)\T, \lambda(t)\T, u(t)\T]\T$ of the boundary-value problem satisfies the first-order necessary condition for $\cal C$ to attain a minimum. As in sensitivity analyses, the costates $\lambda$ may be interpreted as {\em Lagrange multipliers}.

It is customary to introduce the function $H(x,\lambda,u,t) = \lambda\T f(x,u,t)$ (pseudo-Hamilton\-ian) so that (\ref{eq:odecontrol}), (\ref{eq:controladjoint})--(\ref{eq:constraints}) take the very symmetric form
\begin{equation}\label{eq:pseudo}
\frac{d}{dt} x = \nabla_\lambda H,\quad \frac{d}{dt} \lambda = -\nabla_x H, \quad \nabla_u H =0.
\end{equation}

\subsection{The discrete problem: indirect approach}
In the indirect approach, approximations to the optimal states, costates and controls are obtained by discretisation of the boundary value problem  (\ref{eq:odecontrol}), (\ref{eq:controladjoint})--(\ref{eq:constraints}), (\ref{eq:conditions}). Note that we have to tackle  a {\em differential-algebraic} system \cite[Chapter VI.1]{hw}, with the controls being algebraic variables as
$(d/dt)u$ does not feature in any of the equations (\ref{eq:odecontrol}), (\ref{eq:controladjoint})--(\ref{eq:constraints}). Under suitable technical assumptions (invertibility of the second derivative of $H$ with respect to $u$), the  system  is of {\em index one}. This means that the constraints (\ref{eq:constraints}) may be used to express, locally around the solution of interest, the algebraic variables as functions of the differential variables,  $u = \Phi(x,\lambda,t)$. (When applying the implicit function theorem, the relevant Jacobian matrix is the Hessian $\partial_{uu} H$ and this will generically be  positive definite,  if Pontryagin's  principle  \cite[Section 7.2]{trelat} holds so that $H(x,\lambda,\cdot,t)$ is minimised by $\Phi(x,\lambda,t)$.) For a system of index one we may  think that  the right-hand sides of
(\ref{eq:odecontrol}) and (\ref{eq:controladjoint}) have been written as functions of $x$, $\lambda$ and $t$ by setting  $u = \Phi(x,\lambda,t)$, thus transforming the differential-algebraic system into a  system of ordinary differential equations. In fact the transformed system is the canonical Hamiltonian system with Hamiltonian function ${\cal H}(x,\lambda,t) = H(x,\lambda,\Phi(x,\lambda,t),t)$, because the chain rule and
$\nabla_u H=0$ imply that, in (\ref{eq:pseudo}),
$\nabla_x H(x,\lambda,u,t) = \nabla_x {\cal H}(x,\lambda,t)$ and $\nabla_x H(x,\lambda,u,t) = \nabla_x {\cal H}(x,\lambda,t)$.
 This Hamiltonian system may be discretised with the PRK scheme (\ref{eq:rkabc}), (\ref{eq:rkABC}). (Recall that RK schemes are included as particular cases where both sets of coefficients just coincide.) The discrete equations are solved to find the approximations $x_n$ and $\lambda_n$ to $x(t_n)$, $\lambda(t_n)$ and finally the approximations to the controls are retrieved as
$u_n=\Phi(x_n,\lambda_n,t_n)$.

The analytic expression of the implicit function $\Phi$ will in general not be available, so that it will not be possible to find $\cal H$ explicitly. This is not a hindrance: the approximations $x_n$, $\lambda_n$, $u_n$ that one would get by a PRK integration of the Hamiltonian system may  be found in practice as solutions of the  set of
equations (\ref{eq:aug1})--(\ref{eq:aug6})  below, obtained by direct discretisation of the differential-algebraic format (\ref{eq:odecontrol}), (\ref{eq:controladjoint})--(\ref{eq:constraints}). The equivalence between the two approaches, differential and differential-algebraic
is seen by eliminating the controls from (\ref{eq:aug1})--(\ref{eq:aug6}), see \cite[Chapter VI.1]{hw}.

The discrete equations are ($n=0,\dots,N-1$):
\begin{eqnarray}
\label{eq:aug1}
&&x_{n+1} = x_n +h_n \sum_{i=1}^s b_i{k}_{n,i} , \\
\label{eq:n10}
&&{k}_{n,i} = f(X_{n,i},U_{n,i}, t_n+c_ih_n), \quad i=1,\dots, s,\\
\label{eq:aug2}
&&X_{n,i} = x_n + h_n \sum_{j=1}^s a_{ij} {k}_{n,j}, \quad i=1,\dots, s,\\
\label{eq:aug3}
&&\lambda_{n+1} = \lambda_n +h_n \sum_{i=1}^s B_i \ell_{n,i}, \\
\label{eq:n11}
&&\ell_{n,i} = -\partial_xf(X_{n,i},U_{n,i}, t_n+C_ih_n)\T \Lambda_{n,i}, \quad i=1,\dots, s,\\
\label{eq:aug4}
&&\Lambda_{n,i} = \lambda_n + h_n \sum_{j=1}^s A_{ij}\ell_{n,j}, \quad i=1,\dots, s,\\
\label{eq:aug5}
&&\partial_u f(X_{n,i},U_{n,i}, t_n+C_ih_n)\T \Lambda_{n,i} = 0, \quad i=1,\dots, s,
\end{eqnarray}
together with  ($n=0,\dots,N$)
\begin{equation}\label{eq:aug6}
\partial_u f(x_{n},u_{n},t_n)\T \lambda_n = 0,
\end{equation}
 and the boundary conditions $x_0= \alpha$, $\lambda_N=\nabla C(x_N)$ from (\ref{eq:conditions}).

What is the accuracy of this technique? We encounter the same difficulty we found in the preceding section: relevant here is the order of the overall PRK scheme rather than  the (possibly higher) order of the RK coefficients (\ref{eq:rkabc}) used for the state variables.
In the preceding section the approximations $x_n$ are found independently of the $\lambda_n$ and, accordingly, the possible order reduction does not affect them. In the optimal control problem,  states and costates are coupled and any order reduction will harm both of them. This was first noted by Hager  who also provided relevant counterexamples, see \cite[Table 3]{hager}. Hager (Proposition 6.1) also shows that there is no order reduction for explicit, fourth order RK schemes with positive weights.

 The obvious analogue of Theorem \ref{th:variational} holds: the variations $\delta_n$ in the discrete solution $x_n$ satisfy the equations that result from discretising (\ref{eq:varcontrol}) with the coefficients (\ref{eq:rkabc}). These equations are (\ref{eq:rkstepvar}) and (\ref{eq:rkstagesvar}) where now
 \begin{equation}\label{eq:sabado}
  k_{n,i} =\partial_x f(X_{n,i},U_{n,i},t_n+c_ih_n)\, \Delta_{n,i} + \partial_u f(X_{n,i},U_{n,i},t_n+c_ih_n)\, Z_{n,i},
 \end{equation}
 ($\Delta_{n,i}$, $Z_{n,i}$ are the stages associated with the variables $\delta$ and $\zeta$).

 Assume next that the PRK is {\em symplectic}. Recall that symplecticness may be the result of choosing the RK coefficients (\ref{eq:rkabc}) ($b_i\neq 0$, $i=1,\dots,s$) for the state variables and retrieving from (\ref{eq:trick}) the coefficients (\ref{eq:rkABC}) for the integration of the adjoint system. The symplecticness of the integrator  makes it possible to formulate a discrete analogue of Proposition \ref{propcontrol}.

 \begin{theorem}\label{th:main} Assume that $x_n$, $\lambda_n$, $u_n$, $n=0,\dots,N$, satisfy the equations (\ref{eq:aug1})--(\ref{eq:aug6}) arising from the application of a {\em symplectic} PRK method  and that, furthermore, $\delta_n$, $n=0,\dots,N$, $\delta_0=0$, are the variations in $x_n$.   Then, for $n=0,\dots,N-1$,
 $$
 \lambda_{n+1}\T\delta_{n+1} = \lambda_n\T\delta_n.
 $$
The PRK scheme may be  a symplectic RK scheme or the result of choosing freely the RK coefficients (\ref{eq:rkabc}), $b_i\neq 0$, $i=1,\dots,s$, for the states and then using (\ref{eq:trick}) to determine the coefficients for the integration of the costates.
 \end{theorem}
 \begin{proof} Use Lemma \ref{lemma}
with $S(q,p) = \lambda\T\delta$.
This results in
$$
 \lambda_{n+1}\T\delta_{n+1} - \lambda_n\T\delta_n = h_n \sum_i b_i (\Lambda_{n,i}\T k_{n,i}+\ell_{n,i}\T \Delta_{n,i})
 $$
 where $k_{n,i}$ and $\ell_{n,i}$ come from (\ref{eq:sabado}) and (\ref{eq:n11}) respectively.
According to (\ref{eq:identity2}), each of the terms being summed vanishes.
\end{proof}\medskip

When the boundary conditions (\ref{eq:conditions}) are imposed,
 $$
 \nabla {\cal C}(x_N)\T\delta_N = \lambda_N\T\delta_N= \lambda_0\T\delta_0 = 0,
 $$
 which means that the discrete solution satisfies the  first-order necessary conditions for ${\cal C}(x_N)$ to achieve a minimum subject to the constraints (\ref{eq:aug1})--(\ref{eq:aug2}) and $x_0=\alpha$. In this way we have proved that {\em symplectic discretisation} commutes \cite{ross} with the process of forming necessary conditions for minimisation:
 \begin{theorem} \label{th:main2} Let  $\{x_n\}$, $\{\lambda_n\}$, $\{u_n\}$ be a solution of the equations (\ref{eq:aug1})--(\ref{eq:aug6}) arising from discretising with a {\em symplectic} PRK integrator the necessary conditions for the continuous optimal control problem. Then $\{x_n\}$, $\{\lambda_n\}$, $\{u_n\}$  satisfies the necessary conditions for ${\cal C}(x_N)$ to  achieve a minimum subject to the discrete constraints
  (\ref{eq:aug1})--(\ref{eq:aug2}) and $x_0=\alpha$. The PRK scheme may be  a symplectic RK scheme or the result of choosing freely the RK coefficients (\ref{eq:rkabc}), $b_i\neq 0$, $i=1,\dots,s$, for the states and then using (\ref{eq:trick}) to determine the coefficients for the integration of the costates.
 \end{theorem}

 When the states+costates system is integrated by means of a non-symplectic PRK, $x_N$ will  not satisfy the necessary conditions for  $\cal C$ to be minimised subject to  the  constraints (\ref{eq:aug1})--(\ref{eq:aug2}) and $x_0=\alpha$. Therefore non-symplectric PRK discretisations {\em cannot} be obtained via the direct approach considered next.
\subsection{The discrete problem: direct approach}
The direct approach (see e.g. \cite[Chapter 9]{trelat}) based on RK discretisation begins by applying the  scheme (\ref{eq:rkabc})
to the differential equation (\ref{eq:odecontrol})  to get (\ref{eq:aug1})--(\ref{eq:aug2}). Then, these equations and  $x_0=\alpha$ are seen as constraints of a finite-dimensional optimisation problem for the minimisation of ${\cal C}(x_N)$.

We use the standard method of Lagrange multipliers based on  the Lagrangian in (\ref{eq:hat}), trivially adapted to the present circumstances by letting $f$ depend on the controls.
The method leads in a straightforward way to the following result, first proved by
 Hager \cite{hager}, see also \cite{numer}. However \cite{hager} does not point out  that the relations (\ref{eq:trick}) correspond to symplecticness. Furthermore \cite{hager} and \cite{numer} do not use a discrete Lagrangian obtained by discretisation of the continuous Lagrangian. These papers and \cite{vilmart} do not point out that the occurrence of symplectic schemes in this context is really due to the conservation property (\ref{eq:prop}).

\begin{theorem}\label{th:direct}
The first-order necessary conditions for the  minimisation of ${\cal C}(x_N)$ subject to $x_0=\alpha$ and (\ref{eq:aug1})--(\ref{eq:aug2}), $b_i\neq 0$, $i=1,\dots,s$,
are $x_0= \alpha$, $\nabla {\cal C}(x_N) = \lambda_N$  together with (\ref{eq:aug1})--(\ref{eq:aug5}), with the coefficients $A_{ij}$, $B_i$, $C_i$ given by (\ref{eq:trick}).
\end{theorem}

In other words, when the direct approach is used, we arrive at {\em exactly the same set of equations} for $x_n$, $\lambda_n$, $X_{n,i}$, $\Lambda_{n,i}$, $U_{n,i}$ we obtained, with the help of RK technology, via the indirect approach in Theorem \ref{th:main2}. Let us observe that the direct approach does not provide \lq natural\rq\ approximations $u_n$ to $u(t_n)$. Hager \cite{hager} suggests to define $u_n$ by locally minimising $H(x_n,\lambda_n, u,t_n)$ which leads to  (\ref{eq:aug6}). He also notes (\cite{hager}, Table 4) that the order of convergence of the control stages $U_{n,i}$ may be lower than that in $u_n$, something that it is not surprising at all: typically, internal stages are less accurate than end-of-step approximations. We remark that, in the direct approach and once the RK method for $x$ has been chosen, the minimisation of $\cal C$  implicitly provides the \lq right\rq\  coefficients $A_{ij}$, $B_i$, $C_i$ to be used in the integration of the costates in order to ensure symplecticness of the overall PRK integrator. In the indirect approach those coefficients have to be determined by using the relations (\ref{eq:sympcond2})--(\ref{eq:sympcond2c}) and Theorem \ref{th:cooper2}.

While the direct and indirect approaches may be seen as mathematically equivalent here, both have their own interest. The direct approach  suggests to solve the discrete PRK equations with the help of optimisation techniques and these may be an efficient choice in practice. On the other hand, the direct approach \lq hides\rq\ the PRK integration of the costates, a fact that may lead to the false impression that the order of accuracy of the overall procedure coincides with the order of the RK scheme used to discretise the differential constraint (\ref{eq:odecontrol}). This was emphasised in \cite{hager}, where the order of the PRK method (\ref{eq:rkabc}), (\ref{eq:rkABC}), (\ref{eq:trick}) is called the order of the RK method (\ref{eq:rkabc}) {\em for optimal control problems}. A  discussion of the advantages of the direct and indirect approaches is not within our scope here, see e.g. \cite[Chapter 9]{trelat}, \cite{zuazua}.

\section{Some extensions}
\label{sec:exten}

We now consider more general optimal control problems. We shall need to generalize
Theorems \ref{th:cooper} and \ref{th:cooper2} to the situation where the  quantities $I$ or $S$ are not constant along trajectories of the system but vary in a known manner.

\subsection{Generalised conservation}
Here are simple generalisations of Theorems \ref{th:cooper} and \ref{th:cooper2}. Only Theorem \ref{th:cooper4} will be proved; the other proof is  very similar.

 In order to better understand Theorem \ref{th:cooper3}, we may look at the case where $y$ comprises positions and velocities of a mechanical system and $I$ is the kinetic energy. Conservation of energy demands that the rate of change of $I$ coincides with the rate of change (power)  $\varphi$ of the work of the forces. Along each trajectory, the gain in kinetic energy exactly matches the total work exerted by the forces.
\begin{theorem}\label{th:cooper3}Assume that, for the differential system (\ref{eq:ode}), there exist a real-valued bilinear mapping $I$  in $\mathbb{R}^D\times\mathbb{R}^D$ and a real-valued function $\varphi$ in $\mathbb{R}^D$ such that, for each solution $y(t)$
$$
\frac{d}{dt} I(y(t),y(t)) = \varphi(y(t))
$$
and, therefore,
$$
I(y(t_0+T),y(t_0+T))-I(y(t_0),y(t_0)) = \int_{t_0}^{t_0+T} \varphi(y(t))\, dt.
$$
If the system is integrated by means of a symplectic RK scheme as in (\ref{eq:rkstep})--(\ref{eq:rkstages}), then
$$
I(y_N,y_N) -I(y_0,y_0) = \sum_{n=0}^{N-1} h_n \sum_{i=1}^s b_i\, \varphi(Y_{n,i}).
$$
\end{theorem}

Note that the last sum, based on the RK quadrature weights $b_i$ and in the approximation $y(t_n+c_ih_n) \approx Y_{n,i}$, is the \lq natural\rq\ RK discretisation of the corresponding integral.
\begin{theorem}\label{th:cooper4}Assume that, for the partitioned system (\ref{eq:pode}), there exist a real-valued bilinear map $S$ in $\mathbb{R}^{D-d}\times\mathbb{R}^d$
and a real-valued function $\varphi$ in $\mathbb{R}^{D-d}\times\mathbb{R}^d$, such that for each solution
$$
\frac{d}{dt} S(q(t),p(t)) = \varphi(q(t),p(t))
$$
and, therefore,
$$
S(q(t_0+T),p(t_0+T)) - S(q(t_0),p(t_0)) = \int_{t_0}^{t_0+T} \varphi(q(t),p(t))\, dt.
$$
If the system is integrated by means of a  symplectic PRK scheme as in
(\ref{eq:prkstep1})--(\ref{eq:prkstages1}), then
$$
S(q_N,p_N) -S(q_0,p_0) = \sum_{n=0}^{N-1} h_n \sum_{i=1}^s b_i \,\varphi(Q_{n,i},P_{n,i}).
$$
\end{theorem}
\begin{proof} Use Lemma \ref{lemma} and note that, under the present hypotheses,
$$
S(k_{n,i},P_{n,i})+S(Q_{n,i},\ell_{n,i})=\varphi(Q_{n,i},P_{n,i}),
$$
because $S(f(q,p,t),p)+S(q,g(q,p,t))\equiv \varphi(q,p)$ (cf. the proof of Theorem \ref{th:cooper2}).
\end{proof}

\subsection{Other optimal control problems}
Consider first the situation in Section \ref{sec:control}, but assume that the value $x(t_0)$ is not prescribed. Then $\delta(t_0)$ is free and for (\ref{eq:ander}) to hold it is necessary to impose the condition $\lambda(t_0)=0$. This replaces in (\ref{eq:conditions}) the initial condition $x(t_0)=\alpha$. The results in Section \ref{sec:control} are valid in this setting after the obvious modifications.

We next look at the case where (\ref{eq:odecontrol}) and $x(0) = \alpha$ are imposed, but the cost function is given by
\begin{equation}\label{eq:newcost}
{\cal C}(x(t_0+T))+\int_{t_0}^{t_0+T} {\cal D}(x(t),u(t),t)\, dt
\end{equation}
(this is often called a Mayer-Lagrange cost \cite{trelat}, as distinct from the Mayer cost ${\cal C}(x(t_0+T))$ envisaged before).  The adjoint system and constraints are, respectively,
\begin{eqnarray*}
&&\frac{d}{dt} \lambda = - \partial_x f(x,u,t)\T\, \lambda-\nabla_x {\cal D}(x,u,t),\\
&&\partial_u f(x,u,t)\T \lambda + \nabla_u {\cal D}(x,u,t)= 0.
\end{eqnarray*}
These are of the form in (\ref{eq:pseudo}) for the pseudo-Hamiltonian $H = \lambda^T f + {\cal D}$.

The conservation property (\ref{eq:prop}) is replaced by the generalised conservation
formula
\begin{eqnarray*}
&&\lambda(t_0+T)\T \delta(t_0+T) -\lambda(t_0)\T \delta(t_0) \\ &&\qquad{}+\int_{t_0}^{t_0+T}  \Big(\nabla_x{\cal D}(x(t),u(t),t)\T \delta(t)+\nabla_u
{\cal D}(x(t),u(t),t)\T \zeta(t)\Big)\, dt=0,
\end{eqnarray*}
which holds for arbitrary  $\delta(t)$, $\lambda(t)$ satisfying  the variational equations (\ref{eq:varcontrol}), the adjoint system and the constraints. After setting $\delta(t_0) =0$ and $\lambda(t_0+T) =\nabla {\cal C}(x(t_0+T))$, the generalised conservation formula expresses that the the variation of the cost vanishes, i.e. that the first-order necessary conditions for the minimisation hold.

For a symplectic PRK discretisation of the algebraic-differential system,
Lemma \ref{lemma} may be used, just as in the proof of Theorem \ref{th:cooper4}, to show
(the notation should be clear by now):
\begin{eqnarray*}
&&\lambda_N\T\delta_N -\lambda_0\T\delta_0 + \sum_{n=0}^{N-1}h_n \sum_{i=1}^s b_i \Big(
\nabla_x {\cal D}(X_{n,i},U_{n,i},t_n+c_ih_n)\T \Delta_{n,i}\\&&\qquad\qquad\qquad\qquad\qquad\qquad\qquad{}+
\nabla_u{\cal  D}(X_{n,i},U_{n,i},t_n+c_ih_n)\T Z_{n,i}
 \Big)=0.
\end{eqnarray*}
By setting $\lambda_N= \nabla {\cal C}(x_N)$ and $\delta_0=0$, this formula expresses the necessary condition (orthogonality between gradient and variantion) for the discrete solution to minimise the discretised cost
$$
{\cal C}(x_N) +  \sum_{n=0}^{N-1}h_n \sum_{i=1}^s b_i {\cal D}(X_{n,i},U_{n,i}).
$$
Therefore also in this case, results corresponding to Theorems \ref{th:main2} and \ref{th:direct} hold  for a symplectic PRK discretisation.

It is of course possible to combine the cost (\ref{eq:newcost}) with alternative boundary specifications. If $x(t_0)$ is not prescribed, then we have to impose $\lambda(t_0) = 0$, as pointed out above. If both $x(t_0)=\alpha$ and $x(t_0+T) = \beta$ are imposed (in which case the term ${\cal C}(x(t_0+T))$ may be dropped from the cost), then $\lambda(t_0)$ and $\lambda(t_0+T)$ are both free.

\subsection{Constrained controls}

Let us go back once more to the problem in Section \ref{sec:control} and suppose  that the controls $u$ are constrained so that, for each $t$, it is demanded that $u(t)\in U$, where $U$ is a given closed, convex subset of $\mathbb{R}^\nu$. Then (see  e.g. \cite{hager}), the constraint (\ref{eq:constraints}) on $\lambda$ has to be replaced by
$$
u(t)\in U,\qquad  - \partial_u f(x(t),u(t),t)\T \lambda(t)\in N_U(u(t)),
$$
where $N_U(u)$ is the cone of all vectors $w\in\mathbb{R}^\nu$ such that, for each $v\in U$, $w\T(v-u)\leq 0$. Proceeding as in  Proposition \ref{propcontrol}, we see that now
$(d/dt) \lambda(t)\T\delta(t)\geq 0$ and therefore
$$
\nabla {\cal C}(x(t_0+T))\T \delta(t_0+T) \geq 0,
$$
which is the necessary condition for a minimum in the continuous problem. For a PRK discretisation of the boundary value for the states+costates system, the relation $$(d/dt) \lambda(t)\T\delta(t)\geq 0$$ implies
$$
k_{n,i}\T\Lambda_{n,i}+\Delta_{n,i}\T\ell_{n,i}\geq 0
$$
and therefore we may use  Lemma \ref{lemma} yet again to conclude that for symplectic PRK methods and if {\em the weights $b_i$ are positive},
$$
\nabla {\cal C}(x_N)\T \delta_N \geq 0.
$$
Once more, results similar to Theorems \ref{th:main2} and \ref{th:direct} hold. See \cite{hager2} for   order reduction results.

\section{Lagrangian mechanics}
\label{s:mech}
Let us now consider Lagrangian mechanical systems \cite{arnold}. Denote by ${\cal L}(x,u,t)$ the Lagrangian function, where $x\in\mathbb{R}^d$ are the Lagrangian co-ordinates and $u=(d/dt) x$ the corresponding  velocities. According to Hamilton's principle, the trajectories
$t\mapsto x(t)$ of the system are characterised by the fact that they render stationary (often minimum) the action integral
$$
\int_{t_0}^{t_0+T} \!\!{\cal L}(x(t),u(t),t)\, dt,
$$
among all curves $t\mapsto \bar x(t)$ with $\bar x(t_0) = x(t_0)$ and $\bar x(t_0+T) = x(t_0+T)$. This may of course be viewed as a control problem to make stationary (or even maximum) the cost (\ref{eq:newcost}) with ${\cal C} \equiv 0$ and ${\cal D} = -{\cal L}$, subject to the constraint $\dot x = u$
with fixed end-values $x(t_0)$ and $x(t_0+T)$. The theory in Section \ref{sec:exten} applies.
The pseudo-Hamiltonian is $H(x,\lambda,u,t)= \lambda\T u-{\cal L}(x,u,t)$. The constraint $\nabla_u H = 0$ reads $\lambda = \nabla_u {\cal L}(x,u,t)$; thus the control costates coincide with the mechanical momenta. The elimination of the controls with the help of Pontryagin's principle would determine $u$ as a function $\Phi(x,\lambda,t)$ by maximising (recall that we are here trying to maximise the cost!) the function
$u\mapsto H(x,\lambda,u,t)$. In mechanics, this  exactly corresponds with the
theory of the Legendre transformation as presented in \cite[Section 14]{arnold}: that theory shows that, if ${\cal L}$ is a strictly convex function of $u$, then, at given $x$ and $t$, the velocity vector $u$ that corresponds to a given value of the momentum $\lambda$ is globally uniquely defined and maximises $\lambda\T u-{\cal L}(x,u,t)$. In most mechanical problems ${\cal L} = {\cal T}(x,u,t)-{\cal V}(x,t)$, with $\cal T$ and $\cal V$ the kinetic and potential energy respectively, and  $\cal T$ is quadratic, positive-definite as a function of $u$, thus ensuring the required convexity. In control theory the elimination of the controls $u$ in the pseudo-Hamiltonian $H$ gives rise to the \lq control\rq\ Hamiltonian $\cal H$; correspondingly, in mechanics
the Hamiltonian is defined as the result of expressing in $\lambda\T u-{\cal L}(x,u,t)$ the velocities as functions of the momenta (and $x$ and $t$).
Finally the evolution of the states and costates (mechanical co-ordinates and momenta) obeys Hamilton's canonical equations. Hamiltonian solution flows are symplectic and, in this way, we have travelled all the way from action minimisation to symplecticness.

 A similar journey may take place in the discrete realm. Choose any RK scheme (\ref{eq:rkabc}) with nonzero weights to discretise the differential constraint $(d/dt) x = u$ and minimise  the associated discrete action
$$
\sum_{n=0}^{N-1} h_n \sum_{i=1}^s b_i\, {\cal L}(X_{n,i},U_{n,i}, t_n+c_ih_n).
$$
As we know from  Theorem \ref{th:main2}, this direct approach implies a symplectic PRK integration of the Hamiltonian system for $x$ and $\lambda$, where the $\lambda$ equations are integrated with the coefficients (\ref{eq:rkABC}). This is nothing more than the variational construction of PRK symplectic integrators, already presented in the early paper \cite{suris2} by Suris (see \cite{marsden} for more information on integrators based on the principle of least action, cf. \cite{lall}). In this way, Hager's result \cite{hager} may be viewed as an extension of Suris's work to general control problems.

\section{What is the adjoint of a Runge-Kutta method? Reflecting and transposing  coefficients}
\label{s:scherer}

In this section we examine the relations between the preceding material and the notion of the adjoint of an RK method.

Scherer and T\"{u}rke \cite{scherer} associated with the set of RK coefficients (\ref{eq:rkabc}) two new sets called the reflection and the transposition of the original. The reflected coefficients are given by ($i,j= 1,\dots,s$)
 $$
 a_{ij}^r = b_j-a_{ij},\quad b_i^r = b_i, \quad c_i^r = 1-c_i
 $$
 and the transposed coefficients are defined, only for methods with nonzero weights $b_i$, by
 $$
 a_{ij}^t = b_ja_{ji}/b_i,\quad b_i^t = b_i, \quad c_i^t = 1-c_i.
 $$
 The operations of reflection and transposition commute: the transposition of the reflection coincides with the reflection of the transposition as both lead to
 $$
 a_{ij}^{rt} = b_j-b_ja_{ji}/b_i, \quad b_i^{rt}=b_i,\quad  c_i^{rt} = c_i.
 $$
 Furthermore both operations are involutions: each is its own inverse.

The paper \cite{scherer} introduces the operations of reflection and transposition as algebraic manipulations that make it possible to interrelate important families of RK methods; no attempt is made there to interpret computationally the meaning of integrating with the reflected or transposed coefficients. What do reflection and transposition mean?
The interpretation of reflection is well known \cite[Section 3.6]{ssc},  \cite[Chapter II, Theorem 8.3]{hnw}: a step of length $-h_n$ with the reflected RK method inverts the transformation $y_n\mapsto y_{n+1}$ induced by a step of length $h_n$ with the original method. In this paper we have seen this idea at work when moving from (\ref{eq:rksteplambda})--(\ref{eq:rkstageslambda}) to (\ref{eq:rksteplambdabis})--(\ref{eq:rkstageslambdabis}). The formulas (\ref{eq:trick}) provide meaning to the idea of transposition: to construct a symplectic PRK out of a given RK method with nonvanishing weights the $p$ coefficients  are determined by reflecting and transposing the given $q$ coefficients. The transposed of the $q$ coefficients are then those required to integrate backwards the $p$ equations in, say, sensitivity analyses.

As a further illustration of these ideas, consider the linear non-autonomous system
$$
\frac{d}{dt} q = M(t) q,\quad \frac{d}{dt} p = - M(t)\T p,
$$
integrated with the PRK method (\ref{eq:rkabc}), (\ref{eq:rkABC}) (this is a Hamiltonian system). Since $p$ and $q$ are uncoupled, this amounts to an RK integration of the $q$ equations with the coefficients (\ref{eq:rkabc}) together with an RK integration of the $p$ equations with the coefficients (\ref{eq:rkABC}). The system has the invariant
$q\T p$; Theorem \ref{th:cooper2} ensures that it will be preserved if the $p$ coefficients are the transposition of the reflection of the $q$ coefficients. Both sets of coefficients only coincide if $q$ itself is integrated symplectically. If we wish to preserve the invariant, a nonsymplectic integration of $q$ is possible, but then one has to compensate by integrating the $p$ equations in an appropriate way and the order and stability of the $p$ integration have to be investigated separately. Again,  if the $p$ equations are integrated backward in time, then, preservation of $q\T p$ requires that such backward integration be performed with the transposition of the coefficients used to propagate $q$ forward.

We conclude this section  with a remark on terminology. Monographs such as \cite{hlw} and \cite{ssc} use the word {\em adjoint} to refer to the method with reflected coefficients. Section \ref{s:adjoint} and our last comments suggest that, in order to proceed as in the differential equation case, it would have been better to keep the word {\em adjoint} for the reflected and transposed method. And call {\em reflected} to what in \cite{hlw} or \cite{ssc} is called {\em adjoint}. With that alternative terminology, for RK schemes, symplecticness would simply be {\em self-adjointness}.

\section{Conclusion}
\label{s:conc}
Symplectic RK and PRK schemes preserve, by definition, the symplectic form in phase space; in addition, they may be characterized as those RK or PRK integrators that exactly preserve each quadratic invariant of the system being integrated. In  sensitivity analysis, optimal control and other areas, adjoint systems  are introduced and possess paramount importance; these adjoints are defined so as to preserve the key quadratic invariant (\ref{eq:prop}). Therefore, there are  tight connections between those areas and the theory of symplectic integration; we hope the present paper has helped to understand those connections.

\bigskip

{\bf Acknowledgments.} I am extremely indebted to E. Zuazua for providing me with the initial motivation for this research. He, J. Frank, A. Murua, S. Ober-Bl\"{o}baum and E. Tr\'{e}lat kindly provided  useful references. Additionally, A. Murua has to be thanked for sharing  some of his many insights; in particular in connection with automatic differentiation. The extremely careful reading of the manuscript made by one of the referees is also acknowledged with gratitude. 

\section*{Appendix: Schemes with some vanishing weights}
\label{s:vanishing}

If one or more weights $b_i$ in (\ref{eq:rkabc}) vanish, then it is not possible to use  the recipe (\ref{eq:trick}) to define the coefficients required to create a combined symplectic PRK method (\ref{eq:rkabc}), (\ref{eq:rkABC}). Given the partitioned system (\ref{eq:pode}) and the $q$ coefficients (\ref{eq:rkabc}), how to integrate the $p$ equations so as to have a symplectic scheme? The solution to this problem is rather weird and it is best to
  begin with the simplest example.

 Let us study the second-order scheme (due to Runge in his 1895 original paper \cite[ Section II.1]{hnw}), $s=2$,
 \begin{equation}\label{eq:runge}
 a_{11} = a_{21} = a_{22}=0,\: a_{12} = 1/2,\: b_1= 1, \: b_2 = 0,\: c_1 = 1/2,\: c_2=0.
 \end{equation}
 While it is customary to label the stages so that the abscissas $c_i$ increase with $i$,  we have departed from this practice; if we adopted it,  formula (\ref{eq:fancyp}) below would get a rather disordered appearance.

 We regularise the zero weight and consider the one-parameter family, $\epsilon\neq 0$:
 \begin{equation} \label{eq:01}
 a_{11} = a_{21} = a_{22}=0,\: a_{12} = 1/2,\: b_1= 1, \: b_2 = \epsilon,\: c_1 = 1/2,\: c_2=0.
 \end{equation}
 (The regularised scheme is not even consistent, but this does not hinder the argument.)
 From (\ref{eq:trick}), we set
  \begin{equation} \label{eq:02}
 A_{11} = 1,\:A_{12} = A_{22}=\epsilon,\:A_{21} =1- 1/(2\epsilon),\:  B_1= 1,\: B_2=\epsilon, \: C_1 = 1/2,\: C_2 = 0.
 \end{equation}
 Thus, the PRK specified by (\ref{eq:01})--(\ref{eq:02}) is symplectic for each $\epsilon$. The idea now is to take limits as $\epsilon\rightarrow 0$; the limit integrator, {\em if it exists}, will preserve quadratic invariants and, when applied to Hamiltonian problems, the symplectic structure. The difficulty is that from the equation that defines $P_{n,2}$
$$
P_{n,2} = p_n +h_n\left(1-\frac{1}{2\epsilon}\right)  g(Q_{n,1},P_{n,1}, t_n+h_n/2) + h_n\epsilon\, g(Q_{n,2},P_{n,2},t_n)
$$
we may expect that, for fixed $q_n$, $p_n$, the stage vector $P_{n,2}$ grows unboundedly as $\epsilon \rightarrow 0$ and that, therefore, a limit integrator cannot be defined. However, the stage $P_{n,2}$ only affects  $P_{n,1}$ and $p_{n+1}$ through the {\em small} coefficients $A_{1,2} = B_2 = \epsilon$, and this makes it possible to prove that the limit scheme exists for some particular differential equations. Specifically, we assume in the remainder of this section that in the partitioned differential system (\ref{eq:pode}) being integrated, $f$ and $g$ have the special form
\begin{equation}\label{eq:especial}
f = f(q,t)\qquad g= L(q,t)+M(q,t)p
\end{equation}
(with $q=x$, $p=\lambda$, this format includes the system (\ref{eq:odex}), (\ref{eq:adj}) in Section \ref{s:adjoint}). When (\ref{eq:especial}) holds, the $q$ integration  with coefficients (\ref{eq:01}) converges, as $\epsilon\rightarrow 0$, to the integration with the originally given coefficients
 (\ref{eq:runge}). The system for the $p$ stages $P_{1}$, $P_{2}$ (the index $n$ is sometimes dropped to shorten the formulas) may be written as
 \begin{eqnarray*}
 P_1&=& p_n+h_n (L_1+M_1P_1)+ h_n  (\epsilon L_2 + h_n M_2 m_2),\\
m_2 &=& \frac{\epsilon}{h_n} p_n + \big( \epsilon -\frac{1}{2}\big) (L_1+ M_1 P_{1})+\epsilon(\epsilon L_2+h_nM_2m_2),
 \end{eqnarray*}
 where we have scaled $m_2 = (\epsilon/h_n) P_2$ to avoid blow-up and used the abbreviations
 \begin{eqnarray*}
 &&L_1 = L(Q_{1},t_n+h_n/2),\qquad M_1 = M(Q_1,t_n+h_n/2),\\ &&L_2 = L(Q_{2},t_n),\qquad\quad\qquad\:\,M_2 = M(Q_2,t_n).
 \end{eqnarray*}
 Now take limits as $\epsilon\rightarrow 0$, to get
\begin{eqnarray*}
P_1&=& p_n+h_n (L_1+M_1P_1)+ h_n^2  M_2 m_2,\\
m_2 &=&   -\frac{1}{2} (L_1+M_1P_1).
\end{eqnarray*}
Since $B_1 = A_{11}$ and $B_2= A_{12}$, the end-of-step approximations is given by $p_{n+1} = P_1$.

We write these equations in a way similar to (\ref{eq:prkstep1})--(\ref{eq:prkstages1}):
\begin{eqnarray}\label{eq:rungep}
p_{n+1}&=& p_n+h_n \ell_1 + h_n^2  M_2 m_2,\\
\nonumber \ell_1 &=& g(Q_1, P_1, t_n+h_n/2),\\
\nonumber
M_2 & =& M(Q_2,t_n),\\
\nonumber
P_1&=& p_n+h_n \ell_1 + h_n^2  M_2 m_2,\\
m_2 &=&   -\frac{1}{2} \ell_1.\nonumber
 \end{eqnarray}

The combination of these formulas for $p$ with the scheme (\ref{eq:runge}) for $q$
is a {\em first-order} integrator that conserves quadratic invariants as in Theorem \ref{th:cooper2} and, for Hamiltonian problems, preserves the symplectic structure. Of course the  integrator is {\em not} a PRK method; since $M =\partial_p g$, the formula (\ref{eq:rungep}) is reminiscent of Runge-Kutta methods that use higher derivatives of the solution \cite[Section II.13]{hnw}. (Such high-order derivative methods cannot be symplectic for general problems \cite{non}.) Note
that, while $\ell_1$ is an approximation to the first derivative $(d/dt)p$, the vector $M_2m_2$ has the dimensions of the second derivative $(d^2/dt^2)p$.

Let us now turn to the general case. Assume that in (\ref{eq:rkabc}) the first $r$ weights $b_1$, \dots, $b_r$ do not vanish, while $b_{r+1}= \dots = b_s=0$. The regularisation procedure used for Runge's method leads to the fancy integrator:
\begin{eqnarray}
p_{n+1} &=& p_n+ h_n \sum_{i=1}^r b_i \ell_i +h_n^2 \sum_{\alpha=r+1}^s M_\alpha m_\alpha.\label{eq:fancyp}\\
P_i &=& p_n +h_n \sum_{j=1}^r \big(b_j -\frac{b_ja_{ji}}{b_i}\big)\ell_j\\&&\qquad\qquad+h_n^2 \!\sum_{\beta = r+1}^s \!\!\big(1-\frac{b_ja_{\beta i}}{b_i}\big)M_\beta m_\beta,\qquad i = 1,\dots,r,\nonumber\\ \label{eq:sab}
m_\alpha & = & -\sum_{j=1}^r b_j a_{j\alpha}\ell_j -h_n \sum_{\beta=r+1}^s a_{\beta\alpha} M_\beta m_\beta,\qquad \alpha = r+1,\dots, s.
\end{eqnarray}
Here the $r$ vectors $\ell_i$ are as  in (\ref{eq:new2}), so that the method uses $r$ slopes and  additionally $s-r$ matrices $M_\alpha = M(Q_\alpha,t_n+  c_\alpha h_n)$. From the relations (\ref{eq:sab}) the $m_\alpha$ may be viewed as  functions of the $\ell_i$.

The following result is a consequence of the construction via regularisation:

\begin{theorem}\label{th:appA} Consider partitioned systems of the special format (\ref{eq:especial}), where the $q$ equations are integrated with the RK scheme (\ref{eq:rkabc}), $b_1\neq 0$,\dots, $b_r\neq 0$, $b_{r+1}= \dots = b_s=0$, and the $p$ equations with the formulas in (\ref{eq:fancyp})--(\ref{eq:sab}). If $S(q(t),p(t))$ is a conserved quantity as in Theorem \ref{th:cooper2}, then
$S(q_n,p_n)$ is independent of $n$. If the system is Hamiltonian, then the map $(q_n,p_n)\mapsto (q_{n+1},p_{n+1})$ is symplectic.
\end{theorem}

 With the terminology of Section~\ref{s:scherer}, for systems of the special form (\ref{eq:especial}), the scheme (\ref{eq:fancyp}) may be viewed as the reflected and transposed of (\ref{eq:rkabc}) when this possesses one or more zero weights.

Proofs of  Theorem~\ref{th:appA} that do not rely on taking limits as $\epsilon \rightarrow 0$ are of course possible. For such an alternative proof of the conservation of $S$, we may note that manipulations (not reproduced here) similar to those used to prove Lemma \ref{lemma} show that for the present method, in lieu of (\ref{eq:lemma}), we may write:
\begin{eqnarray*}
S(q_{n+1},p_{n+1}) -S(q_n,p_n) &=& h_n\sum_{i=1}^r b_i \big(S(k_i,P_i)+ S(Q_i,\ell_i)\big)\\&&\quad{} + h_n^2 \sum_{\alpha=r+1}^s   \big(S(k_\alpha,m_\alpha)+S(Q_\alpha,M_\alpha m_\alpha )\big).
\end{eqnarray*}
This is an algebraic identity that does not require  that the system integrated to be conservative. When $S$ is conserved, the first sum vanishes as in the proof of Theorem \ref{th:cooper2}.
For the second sum note that from $S(f(q,t),p)+S(q, L(q,t)+M(q,t)p)\equiv 0$ it follows that $S(f,p)+S(q,Mp)\equiv 0$.

 For the adjoint equations in Section \ref{s:adjoint}, the conclusion of Theorem \ref{th:adjointPRK} holds if the $x$ equations are integrated with a (nonsymplectic) RK method with one or more vanishing weights and the $\lambda$ equations are integrated as in (\ref{eq:fancyp})--(\ref{eq:sab}). Similarly Theorem \ref{th:automatica} holds for a suitable choice of the Lagrangian (details will not be given, but see below).

What is the situation for the control problem in Section \ref{sec:control}? Recall that the corresponding system of {\em differential equations}  is given by (\ref{eq:odecontrol}), (\ref{eq:controladjoint}), where, in the right-hand sides, $u$ has been expressed as $u=\Phi(x,\lambda,t)$. That system of differential equations does {\em not} possess the format (\ref{eq:especial}) for which (\ref{eq:fancyp}) makes sense and, accordingly, we cannot provide analogues to Theorems \ref{th:main} and \ref{th:main2}.

 In order to gain additional insight, let us use the direct approach based on Runge's second order integrator (\ref{eq:runge}).
We define the Lagrangian (compare with (\ref{eq:hat}) and note consistency with (\ref{eq:friday}) due to the factor $h_n^2$):
\begin{eqnarray*}
&& {\cal C}(x_N) - \lambda_0\T(x_0-\alpha)-\sum_{n=0}^{N-1} h_n\lambda_{n+1}\T \Big[\frac{1}{h_n}(x_{n+1} - x_n) -  {k}_{n,1}\Big]\\
&&\qquad\qquad {} - \sum_{n=0}^{N-1} h_n \Lambda_{n}\T\Big[ k_{n,1} -f(X_{n,1}, U_{n,1},t_n+h_n/2)\Big]\\
&&\qquad\qquad{} -  \sum_{n=0}^{N-1} h_n^2 \mu_n\T \Big[k_{n,2}-f(X_{n,2}, U_{n,2},t_n)\Big],
\end{eqnarray*}
where, as on other occasions, the stages $X_{n,1}=x_n+(h_n/2) k_{n,2}$, $X_{n,2}=x_n$ must be seen as known functions of $x_n$ and $k_{n,2}$.
 Taking gradients with respect to $x_n$, $k_{n,1}$, $k_{n,2}$ leads to the necessary conditions
 \begin{eqnarray*}
 \lambda_{n+1} &=& \lambda_n - (\partial_x f(X_{n,1},U_{n,1},t_n+h_n/2))\T \Lambda_{n} \\&& \qquad\qquad\qquad{}-h_n^2 (\partial_x f(X_{n,2},U_{n,2},t_n))\T\mu_n,\\
 \Lambda_{n} &=&\lambda_{n+1}, \\
 \mu_n &=& \frac{1}{2} (\partial_x f(X_{n,1},U_{n,1},t_n+h_n/2))\T\Lambda_n;
 \end{eqnarray*}
 which clearly correspond to the integrator (\ref{eq:rungep}). (By considering the case where $f$ is independent of $u$, this shows that Theorem \ref{th:automatica} holds in this case.) However, taking gradients with respect to $U_{n,1}$ and $U_{n,2}$ yields
 $$
 (\partial_u f(X_{n,1},U_{n,1},t_n+h_n/2))\T \Lambda_n=0,\qquad  (\partial_u f(X_{n,2},U_{n,2},t_n))\T\mu_n=0.
 $$
The second equation is totally meaningless. It cannot be seen as a discretisation of (\ref{eq:constraints}) because $\mu_n$ is not an approximation to the costate $\lambda$; it does not even possess the right dimensions for that to happen. The values of $U_{n,2}$ retrieved from this constraint will have no relation to the true optimal controls. The paper  \cite{hager} nicely illustrates this with an example (see also \cite{hager2}).

Since the trouble arises by the presence of the controls, things may be fixed by tampering with  $U_{n,2}$, as pointed out in \cite{hager}, \cite{hager2}. However, there is no shortage of RK schemes with nonzero (or even positive) weights, so that, in practice, resorting to such fixes seems ill advised.
\end{document}